\documentclass[hidelinks,onefignum,onetabnum]{siamart220329}

\usepackage{lipsum}
\usepackage{physics}
\usepackage{siunitx}
\usepackage{amsfonts}
\usepackage{graphicx}
\usepackage{epstopdf}
\usepackage{esint}
\usepackage{braket}

\usepackage[algo2e,linesnumbered,ruled,vlined]{algorithm2e}
\SetKwInput{KwData}{Input}
\SetKwInput{KwResult}{Output}

\usepackage{todonotes}

\ifpdf
\DeclareGraphicsExtensions{.eps,.pdf,.png,.jpg}
\else
\DeclareGraphicsExtensions{.eps}
\fi

\newsiamremark{remark}{Remark}
\newsiamremark{hypothesis}{Hypothesis}
\crefname{hypothesis}{Hypothesis}{Hypotheses}
\newsiamthm{claim}{Claim}

\headers{Dynamic off-the-grid untangling of curves by Riemannian
metric}{T. Bertrand and B. Laville}

\title{Dynamic off-the-grid untangling of curves with Reeds-Shepp
  metric.\thanks{Submitted to the editors \today.
    \funding{The work of Bastien Laville has been supported by the
      French government, through the UCA DS4H Investments in the Future
      project managed by the National Research Agency (ANR) with the
      reference number
      \href{https://anr.fr/ProjetIA-17-EURE-0004}{ANR-17-EURE-0004}.
      The work of Théo Bertrand has been supported by the French
      government, through the 3IA PRAIRIE Institute, 'Investments in
      the Future' project managed by the National Research Agency (ANR)
      with the reference number
    \href{https://anr.fr/ProjetIA-19-P3IA-0001}{ANR-19-P3IA-0001}.}
}}

\author{Théo Bertrand\thanks{CEREMADE, UMR CNRS 7534, University
    Paris Dauphine, PSL Research University, 75016 Paris, France
  (\email{tbertrand@ceremade.dauphine.fr}). Authors are ordered alphabetically.}
  \and Bastien Laville\thanks{Morpheme team, Inria, CNRS, Université
    Côte d'Azur, Nice, France
    (\email{bastien.laville@pm.me},
    \url{https://www-sop.inria.fr/members/Bastien.Laville/}).
  Corresponding author.}
}

\usepackage{amsopn}

\ifpdf
\hypersetup{
  pdftitle={Dynamic off-the-grid untangling of curves with Reeds-Shepp metric},
  pdfauthor={Bastien Laville, Théo Bertrand},
  pdfsubject={Calculus of variations},
  pdfkeywords={Curve untangling}
}
\fi

\usepackage{notations}
\usepackage{physics, bm, mathrsfs, braket, amssymb, amsmath, amsbsy,
mathtools, enumitem}
\usepackage{bbm}

\usepackage{caption, subcaption}

\begin{document}

\maketitle

\begin{abstract}
  We propose an improved strategy for point sources tracking in a
  temporal stack through an off-the-grid fashion, inspired by the
  Benamou-Brenier regularisation in the literature. We define a
  lifting of the problem in the higher-dimensional space of the
  roto-translation group. This allows us to overcome the theoretical
  limitation of the off-the-grid method towards tangled point source
  trajectories, thus enabling the reconstruction and untangling even
  from the numerical standpoint. We define accordingly a new
  regularisation based on the relaxed Reeds-Shepp metric, an
  approximation of the sub-Riemannian Reeds-Shepp metric, further
  allowing control on the local curvature of the recovered trajectories.
  Then, we derive some properties of the discretisation and prove a
  $\Gamma$-convergence result, fostering interest for practical
  applications of polygonal, Bézier, and piecewise geodesic discretisation.
  We finally test our proposed method on a localisation problem
  example, and give a fair comparison with the state-of-the-art
  off-the-grid method.

\end{abstract}



\section{Introduction and contributions}

\label{sec:intro}
This work focuses on the definition of an off-the-grid framework,
especially with an energy able to track \emph{crossing} point sources
from a temporal stack of measurements. In the context of inverse
problem, this amounts to the recovery of crossing Dirac measures
trajectories, also called dynamic Dirac measures. Such interest is
motivated by the latest refinements in the off-the-grid community
\cite{Bredies2019, Bredies2022, Tovey2021}, where several results,
such as an energy for \emph{moving} point trajectories recovery and
numerical implementation were enabled. It caters for both
\emph{balanced} (constant amplitude over time) and \emph{unbalanced}
(fluctuating amplitude) cases.

However, the literature in calculus of variations nowadays does not
offer a genuine tractable off-the-grid framework for tangled paths:
for instance, if two spikes are crossing, the current
state-of-the-art cannot tackle these reconstructions and, on the
contrary, yield two separate paths, which do not touch in any way.
Yet such situations arise naturally in several domains such as
biomedical imaging, where the physical objects used for
super-resolution such as air bubbles can cross. An example of curves
encountered in a natural and practical setting is presented in Figure
\ref{fig:blood-vessels}.

\begin{figure}[ht!]
  \centering
  \includegraphics[width=0.7\linewidth]{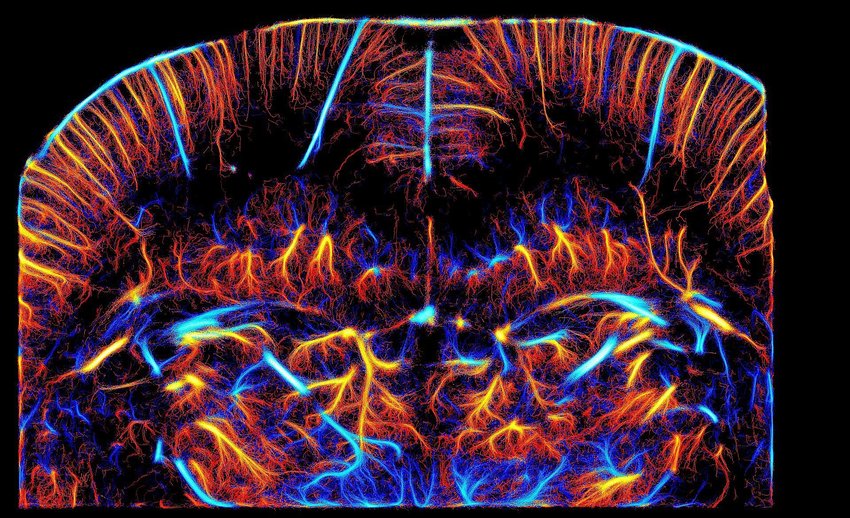}
  \caption{Crossing curves arise genuinely in several biological
    structures, such as cell nucleus, blood vessels, \emph{etc}. This
    image is a reconstruction of blood vessels in a rat brain,
    performed by ULM (Ultrasound Localisation Microscopy) as joint work
    from ESPCI, Inserm and CNRS \cite{Errico2015}. Color relates to the
    velocimetry of the microbubbles, red indicates   high positive
  velocity while blue corresponds to high negative velocities.}
  \label{fig:blood-vessels}
\end{figure}

In the following, we propose a method to recover the point sources
paths, even when a crossing occurs. Our main idea boils down to a
lifting on the roto-translation space  \cite{Chambolle2019,
duits2018optimal}, see Figure \ref{fig:roto-illustration}, while
formulating a new energy with a Riemannian regularisation tailored
for paths untangling. Additionally, we propose a general theoretical
framework for the numerical aspects,  which proves convergence of the
discretised problem towards the continuous problem, with a reduced
set of hypothesis.
To the best knowledge of the authors, this is the first attempt to
recover tangled paths as dynamic Dirac measures, in an off-the-grid manner.

We will consider an inverse problem model similar to the one proposed
in \cite{Schmitzer2019, Bredies2019a} for which \cite{Tovey2021}
added numerous numerical recipes. The main idea to solve this inverse
problem in an off-the-grid fashion boils down to the reconstruction
of a Radon measure supported on a space of curves. The reconstruction
is achieved by minimising an energy functional with a data term that
penalises distance between data and the candidate minimiser, and a
regularisation term that is linear w.r.t. the candidate. Such
formulation dictactes \emph{by design} the shape of the minimisers:
indeed, this problem was derived while having in mind the opportunity
to use a \emph{representer theorem} as the ones proven in
\cite{Boyer2018, Bredies2019b}. It provides a strong result as it
ensures that the shape of the minimiser is determined by the extremal
points of the unit ball of the regulariser; hence being able to
characterise sharply the solution, as an \emph{a priori}, before even
solving the inverse problem.
The \emph{representer theorem} requires some few hypotheses, among
which are the finite-dimensionality of the space on which the
distance to data is computed. In \cite{Bredies2019b}, the authors
give a proof that the extremal points of the unit ball for the
considered Benamou-Brenier regularisation are Dirac measures on the
space of curves, thus ensuring the sparsity and the shape of the solution.
One may note that if \cite{Bredies2019b} and \cite{Bredies2019}
express the problem in terms of a time-dependent density $\rho(x,t)$
(represented as a measure-valued continuous function of time) and a
velocity field $v_t$, coupled via the continuity equation, the
solution realises the optimal transport between two marginal
distributions, hence could be described by a path $\gamma$. Even if
the correspondence is not one-to-one, one can associate a measure
path $\delta_\gamma$ to a density $\rho(x,t)$ and velocity field
$v_t$ couple, and conversely, at least under suitable smoothness conditions.

On the numerical side, a few works have been trying to propose
numerical methods to recover curves from a temporal series of
measurements, or stack of images. For instance, \cite{Bredies2022}
proposes a \emph{conditional gradient method} or \emph{Frank-Wolfe
algorithm} \cite{Frank1956} to solve the considered inverse problem.
\emph{Frank-Wolfe} type algorithms are typical greedy algorithms
deployed on a convex optimisation problem on a compact set, it
consists of iteratively minimising the linearised version of the
objective function on a compact set. Classic theory tells us that
such minimisers will be located on extremal points of this compact
set. The latest developments \cite{Tovey2021} expand on the
Frank-Wolfe algorithm setup and leverage a dynamical programming
approach to speed up the computation. They also extend the numerics
to the 'unbalanced' case, where each curve has an associated
time-dependent amplitude, accounting for mass creation and destruction.

\begin{figure}[ht!]
  \centering
  \begin{subfigure}[b]{0.46\textwidth}
    \centering
    \includegraphics[width=\textwidth]{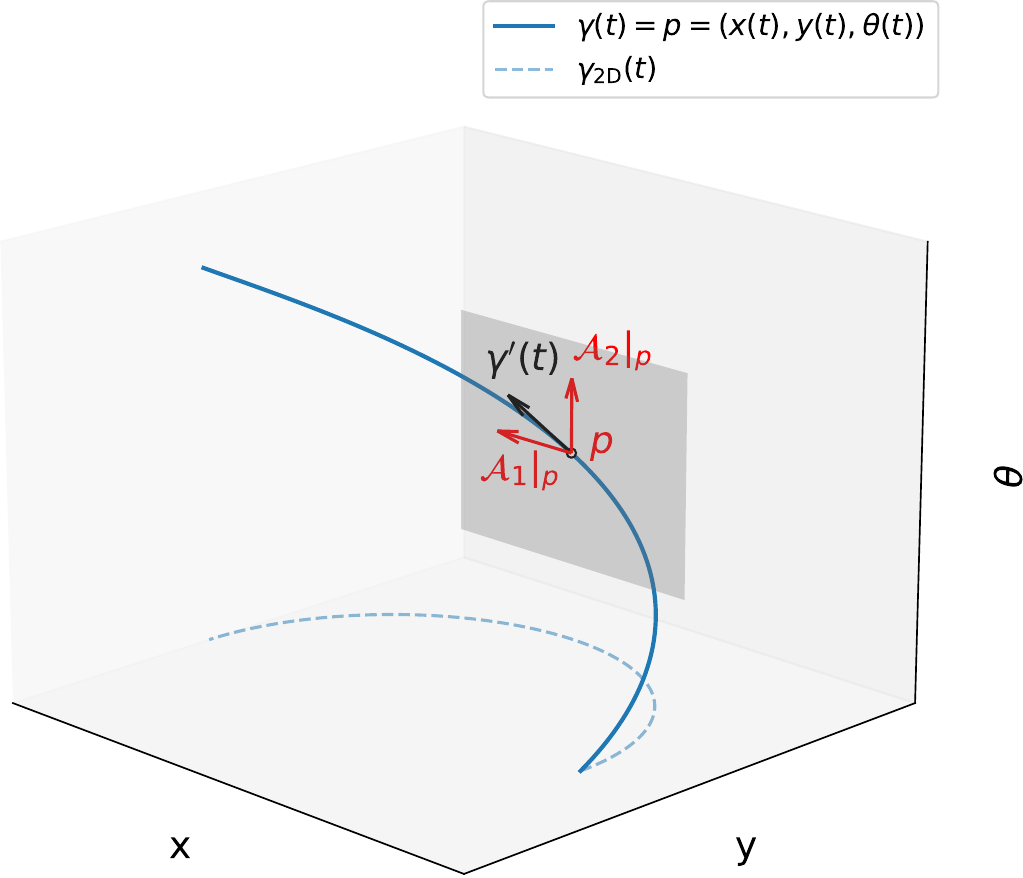}
  \end{subfigure}
  \hspace{0.05\textwidth}
  \begin{subfigure}[b]{0.41\textwidth}
    \centering
    \includegraphics[width=\textwidth]{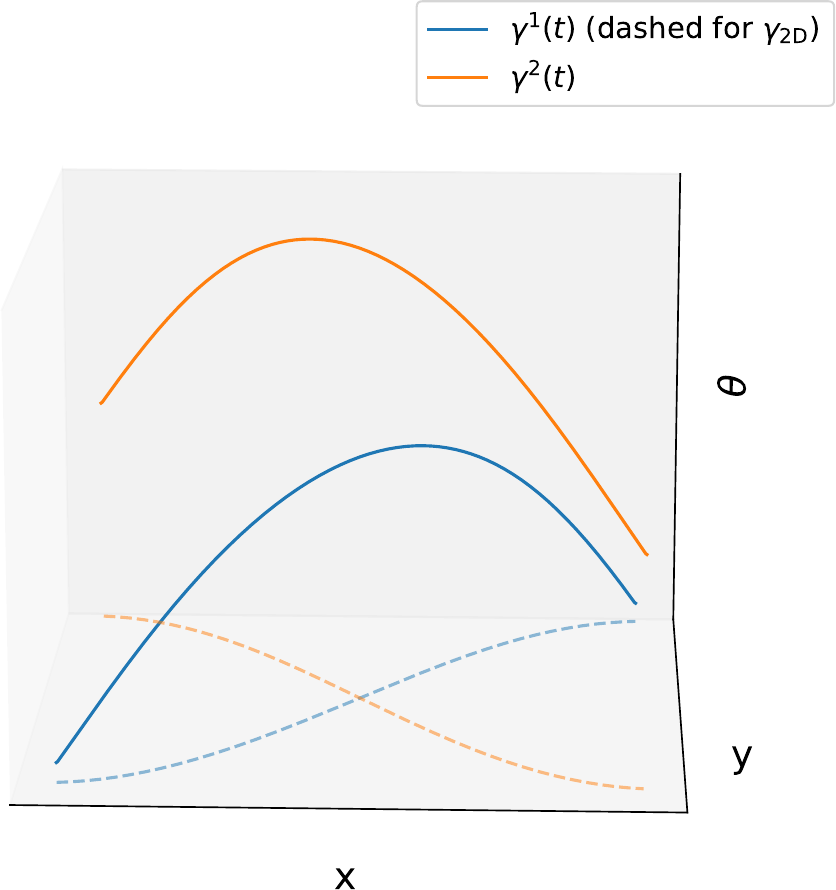}
  \end{subfigure}
  \caption{(Left) Each point on the planar curve $\gamma_\text{2D}(t)
    = (x(t), y(t))$ is lifted to a point $p = \gamma(t) = (x(t), y(t),
    \theta(t)) \in \rototrans_2$ along a horizontal curve (solid line),
    where $\theta(t)$ denotes the direction of the tangent vector
    $\dot{\gamma}_{2D}(t)$ of the original planar curve. The
    differential structure enables a tangent plane at $p$ spanned by
    $\mathcal A_1|_p$ and $\mathcal  A_2|_p$, where the tangent vector
    $\gamma'(t)$ lies. Adapted from \cite{bekkers_2015}. (Right) The
  two dashed planar curves crossing in 2D are untangled in the lifting space.}
  \label{fig:roto-illustration}
\end{figure}

\subsection{Contributions}
Our main contributions in this article are listed below.
\begin{itemize}
  \item Introduce a lift of the curve recovery problem in the
    roto-translation space, also understood as the position-orientation space;
  \item Define a new regularisation on this space thanks to a
    Riemannian metric tailored for dynamic Dirac measure recovery, in
    particular for crossing trajectories;
  \item Prove the $\Gamma\text{-convergence}$ of the discretised
    curve problem to the global one, thus showing the interest of
    discretising the space of curves using polygonal, Bézier and
    piecewise geodesic curves as approximations in the practical context;
  \item Propose an efficient greedy algorithm, with several numerical
    examples to illustrate the capabilities of our proposed method.
    We study balanced and unbalanced cases, with polygonal, Bézier
    and piecewise geodesic discretisations.

\end{itemize}

\subsection{Paper outline}
This paper is organised as follows: the first section introduces the
subject with as little mathematics and jargon as possible and gives
an idea of what has been done and what elements we bring to the
literature, the second section introduces the variational problem,
then the third section gives theoretical insights as to why the
proposed discretisations work. Finally, the last section shows our
attempts to bring our method to concrete applications, with convincing results.

\subsection{Notations}
Let $d \in \N^*$ be the dimension of the ambient space, namely the
space where the sources positions live. In the following, $\xx$ is a
closed, bounded and convex subset of $\R^d$. $\hilb$ is a Hilbert
space, endowed with its canonical norm $\norm{\cdot}_\hilb$.
$\sphere^d$ is the $d$-sphere, understood as the space of
orientation. The roto-translation space will be further denoted
$\rototrans^d \eqdef \R^d \rtimes \sphere^d$, for $\rtimes$ the
semi-direct product pertaining to the group structure and $\SO(d)$
the $d$-space of rotation. We also note $\wedge$ the cross-product,
and $\cdot$ the dot product.
Let $(X_1,S_1)$ and $(X_2,S_2)$ be two measurable spaces, $m$ be a
measure on $S_1$ and $f: X_1 \to X_2$ a measurable function. The push
forward of $\mu$ is defined as the measure $f_\sharp \mu$ defined for
all $B \in S_2$ by $f_{\sharp}\mu(B)\eqdef \mu(f^{-1}(B))$.
The tensor notation, and especially the Einstein notation will be
used: consider $(x^1, \dots, x^N)$ some contravariant tensor, and a
covariant tensor $(c_1, \dots, c_N)$ then the sum $y = \sum_{i=1}^N
c_i x^i$ is abbreviated $y = c_i x^i$ where we omit the summation symbol.

\section{Dynamic off-the-grid 101: a quick explanation on the setting}

This work belongs to the calculus of variations field, with proposed
applications to inverse problem: an \emph{inverse problem} aims to
recover some physical quantities, from a low-passed noisy
observation. This is typically the localisation of sources, from a
blurred, downgraded image or measurement.
The \emph{off-the-grid} variational methods, also called
\emph{gridless} methods, are a rather recent addition to the
literature \cite{Bredies2012, Castro2012, Duval2014, Denoyelle2016},
designed to overcome the limitations of the discrete methods, more
precisely by the fine grid. Indeed, in the discrete framework such as
the Basis Pursuit/LASSO, a fine grid is introduced and the sources
are estimated within this setting: the point (also called
\emph{spikes}) are thereby constrained on this grid, further yielding
discretisation discrepancies. On the contrary, off-the-grid methods
do not rely on such grid and rather consist in the optimisation of a
measure in terms of amplitude, position and number of the point
sources, it can then fully leverage the physical knowledge of the
structures we aim to retrieve. Moreover, they offer several
theoretical results, such as a quantitative bound characterising the
discrepancies between the source and the reconstruction. In the
following, we propose a quick recall of these notions; we advise the
interested reader to take a glance at the seminal papers cited
before, or at \cite{Laville2021} to get a broader vision of the
domain and its applications.

\subsection{Off-the-grid recovery of point sources in a static manner}

The machinery behind the off-the-grid framework really boils down to
the notion of \emph{measure}. Several measure spaces and
regularisations were proposed in the literature \cite{Bredies2012,
Castro2021, Laville2023, Bredies2019} for distinct geometry of the
sources such as points, sets, curves, \emph{etc}. In the following,
we will investigate for the sake of pedagogy the classic case study
in the gridless methods: the recovery of point sources from one
observation, without any notion of time. As defined in the notations
section, let $\xx$ be a closed, bounded and convex set of $\R^d$.

\subsubsection{Radon measure space}

A spike, or a point source, can \emph{physically} take any position
on the \emph{continuum} $\xx$; it cannot be constrained to a finite
set of positions -- at least in the physical world. It can be
accurately modelled by a Dirac measure/masse $a \delta_x$: loosely
speaking, this map allows us to encode both amplitude $a \in \R$ and
spatial continuous information $x \in \xx$ in the same object.
However, since the Dirac measure is not a classic continuous
function, one needs to consider a more general set of mapping called
the Radon measures.

From a distributional standpoint, it is a subset of the distribution
space $\distr$; the latter being the space of linear forms over the
space of test functions $\cistr$ \ie\ smooth functions (continuous
derivatives of all orders) compactly supported. In fact, it is the
smallest Banach space that contains the Dirac masses.
This functional approach is based on the definition of a measure as a
linear form on a function space, hence we will need:

\begin{definition}[Continuous function on $\xx$]
  We call $\czerh$ the set of continuous functions from $\xx$ to a
  normed vector space $\yy$, endowed with the \emph{supremum} norm
  $\normsupremum{\cdot}$ of functions.
\end{definition}

\begin{definition}[Evanescent continuous function on $\xx$]
  We call $\cevah$ the set of evanescent continuous functions from
  $\xx$ to a normed vector space $\yy$, namely all the continuous map
  $\psi : \xx \to \yy$ such that :

  \begin{align*}
    \forall \varepsilon >0, \exists K \subset \xx \,
    \mathrm{compact}, \quad \sup_{x \in \xx \setminus K } \norm{\psi
    (x)}_\yy \leq \varepsilon.
  \end{align*}
\end{definition}

We  write $\ceva$ when $\yy = \R$. Then,

\begin{definition}[Set of Radon measures]\label{def:radon}
  We denote by $\radon$ the set of real signed Radon measures on
  $\xx$ of finite masses. It is the topological dual of $\ceva$ with
  \emph{supremum} norm $\normsupremum{\cdot}$ by the Riesz-Markov
  representation theorem \cite{Federer1996}.
  Thus, a Radon measure $m \in \radon$ is a continuous linear form on
  functions $f \in \ceva$, with the duality bracket denoted by
  $\crochet{f}{m} = \int_\xx f \ud m$.
\end{definition}

A \emph{signed} measure means that the quantity $\crochet{f}{m}$ is
allowed to be negative, further generalising the notion of
probability, hence positive, measure.
Classic examples of Radon measures are the Lebesgue measure, the
Dirac measure $\delta_z$ centred in $z \in \xx$ namely for all $f \in
\ceva$ one has $\crochet{f}{\delta_z} = f(z)$, \etc.

The Banach space $\radon$ can be endowed with the topology of the
norm, or with the topology of its dual called the \wast\ topology.

\begin{definition}[\Wast\ convergence]
  A sequence of measures $(m)_{n\in \N}$ \wastly\ converges to $m \in
  \radon$, denoted by $m_n \cvet m$, if:
  \begin{align*}
    \forall g \in \cistr, \quad \int_\xx g \ud m_n \xrightarrow[n \to
    +\infty]{} \int_\xx g \ud m.
  \end{align*}
\end{definition}

\subsubsection{Observation model}

Let us introduce $\hilb$ the Hilbert space where the acquired data,
or \emph{observation},  live. In the case of images, we use a
finite-dimensional space of acquisition $\hilb = \hilb_n \eqdef \R^n$
for $n \in \N$. Let $m \in \radon$ be a source measure, we call
\emph{acquisition}, or \emph{observation}, $y \in \hilb$ the result
of the \emph{forward/acquisition map} $\Phi : \radon \to \hilb$
evaluated on $m$, with measurement kernel $\varphi: \xx \to \hilb$
continuous and bounded \cite{Denoyelle2018}:

\begin{align}
  y \eqdef \Phi m = \int_\xx \varphi(x) \ud m(x).
  \label{eq:Phi}
\end{align}

Also note that the forward operator $\Phi$ incorporates a sampling
operation, hence $\hilb = \hilb_n$.
In the following, we impose $\varphi \in \mathscr{C}^2(\xx, \hilb)$
and $\forall q \in \hilb$, the map $x \mapsto
\braket{\varphi(x),q}_{\hilb}$ ought to belong to $\ceva$. Let us
also define the adjoint operator of $\Phi : \radon \to \hilb$ in the
\wast\ topology, namely the map $\opadj : \hilb \to \ceva$, defined
for all $x \in \xx$ and $p \in \hilb$ by $\opadj(p) (x) =
\braket{p,\varphi(x)}_{\hilb}$.
The choice of $\varphi$ and $\hilb$ is dictated by the physical
acquisition process \cite{Laville2021}, with generic measurement
kernels such as convolution, Fourier, Laplace, \etc .

\subsubsection{An off-the-grid functional: the BLASSO}

Consider the source measure $\mzer \eqdef \sum_{i=1}^N a_{0,i}
\delta_{x_{0,i}}$ with amplitudes $\vb{a_0} \in \R^N$ and positions
$\vb{x_0} \in \xx^N$, the sparse spike problem aims to recover this
measure from the acquisition $y \eqdef \Phi \mzer + w$ where $w \in
\hilb$ is an additive noise.
In order to tackle this inverse problem, let us introduce the
following convex functional called BLASSO \cite{Bredies2012,
Castro2012}, standing for Beurling-LASSO:

\begin{align}
  \argmin_{m \in \radon} T_\lambda(m) \eqdef \dfrac{1}{2} \norm{y -
  \Phi(m)}^2_{\mathcal H} + \lambda \mtv,
  \tag{\ensuremath{\mathcal{P}_\lambda (y)}}
  \label{eq:blasso-bruits}
\end{align}

where $\lambda > 0$ is the regularisation parameter accounting for
the trade-off between fidelity and sparsity of the reconstruction.
Then thanks to convex analysis results, one can establish the
existence of solutions to \eqref{eq:blasso-bruits} as proved in
\cite{Bredies2012}. The difficulties also lie in the question of
uniqueness of the solution and correct support recovery: these
questions were addressed in several papers in the literature
\cite{Duval2014, Denoyelle2016}.

A very general result in the calculus of variations community called
the \emph{representer theorem} has been thoroughly applied to
numerous off-the-grid energies, such as the BLASSO. Let $G:\hilb_n
\to \R$ an arbitrary data fitting term; suppose that $\Lambda :
\Omega \to \hilb_n$ is linear and $F:\Omega \to \R$ is convex, and
for $m \in \Omega$ the energy reading $J(m) \eqdef G(\Lambda m) +
F(m)$. An extreme point of a convex set $S$ is a point $x \in S$ that
cannot be expressed as a convex combination of two distinct points in
$S$, we denote by $\Ext(S)$ the set of extreme points of $S$.
The following theorem, due to \cite{Boyer2018, Bredies2019b},
establishes (up to some hypotheses) the link between the minimisers
of the functional $J$ and the extreme points of the unit ball
$\unitball_1^F \eqdef \{ u \in \Omega, F(u) \leq 1 \}$:

\begin{theorem}[Representer theorem]
  If $F$ is semi-norm, there exists $\bar u \in \Omega$, a minimiser
  of $J$ with the representation:
  \begin{align*}
    \bar u  = \sum_{i=1}^p \alpha_i u_i
  \end{align*}
  where $p \leq \dim \hilb_n$, $u_i \in \Ext(\unitball_1^F)$ and
  $\alpha_i > 0$ with $\sum_{i=1}^p \alpha_i = F(\bar u )$.
  \label{th:rpz}
\end{theorem}

This result is significant because it connects the geometry of
minimisers to the structure of the regulariser $R$, and more
precisely, to the extreme points of its unit ball $\unitball_1^F$ (or
sublevel sets). Well-known examples include $\ell_1$ regularisation,
so the BLASSO energy, whose extreme points are Dirac measures. More
recently, a regulariser for divergence vector field/solenoid has also
been proposed that promotes curve-like structures \cite{Laville2023},
albeit within a static framework.
Leveraging the representer theorem is a thriving approach, which
amounts to characterising the structural properties of minimisers
without the need to explicitly solve the optimisation problem
\cite{Boyer2018}. Moreover, understanding the extreme points of the
unit ball of the regulariser is crucial for numerical implementation;
in particular, the Frank–Wolfe algorithm reconstructs a solution by
iteratively adding extreme points of the $R$-unit ball to the current estimate.

Now, we need to expand the BLASSO framework to the \emph{dynamic}
setting, to incorporate some dependence in time. In the latter, the
observation is a stack of images, hence a sequence of images ordered
by time: thereby, the point sources should be \emph{dynamic} and
further carry out a time dependence.

\subsection{Dynamic off-the-grid point sources setting}
\subsubsection{The time cone framework and the push-forward measure
  from \texorpdfstring{ $\mathcal M(\Gamma)$}{} to
\texorpdfstring{$\mathcal M (\Omega)$}{}}

In the seminal papers \cite{Bredies2022,Bredies2019,Bredies2019a} on
off-the-grid dynamic methods, the time-cone $\Omega \eqdef [0,1]
\times \xx $ is introduced as a natural framework for representing
time-dependent sparse measures. The goal of the reconstruction is to
recover multiple points that evolves over time, which can be modelled
as a measure $\rho \in \mathcal  M(\Omega)$, where $\rho(t, x)$
represents the spatial distribution of the measure at time $t$. This
formulation allows the problem to be expressed as a variational
optimisation problem, where one seeks to minimise a cost functional
involving data fidelity and regularisation. To ensure temporal
smoothness, the measure $\rho(t, x)$ is required to satisfy a
continuity equation of the form $\pdv{t} \rho + \nabla \cdot (\rho v)
= 0$, where $v$ is a velocity field governing the transport of mass.
While $\mathcal M(\Omega)$ provides a useful representation of
time-dependent measures, it does not explicitly encode trajectories
of moving sources. To better capture the underlying dynamical
structure, the problem is reformulated in terms of measures on paths,
transitioning from $\mathcal M(\Omega)$ to $\mathcal M(\Gamma)$,
where $\mathcal M(\Gamma)$ represents distributions over smooth
particle trajectories $\Gamma$.

This reformulation allows each point source to be described by a
continuous path $\gamma = (h, \xi)$, where $h(t)$ denotes its mass
and $\xi(t)$ its trajectory over time $t$. By leveraging this
representation, the problem inherently enforces smooth motion
constraints and naturally integrates optimal transport
regularisation. This transformation is justified by measure
disintegration techniques, enabling one to decompose the measure as a
product separating time and spatial components. Moreover, it
establishes a one-to-one correspondence between the solutions of the
continuity equation in $\mathcal M(\Omega)$ and measures on $\mathcal
M(\Gamma)$. From a computational perspective, this approach enables
off-the-grid optimization methods, avoiding the need for space-time
discretisation while facilitating efficient implementation through
graph-based shortest-path algorithms \cite{Tovey2021}.

Thus, working in $\mathcal M(\Gamma)$ leads to a more structured and
computationally tractable formulation of dynamic inverse problems.
In the following section, we give a precision definition of these
spaces and Benamou-Brenier regularisation, ensuring that the
reconstructed point sources follow smooth trajectories over time.
Consequently, since each spike is detected based on the discretised
observation, this penalty reduces the set of possible paths to only
those that are smooth and physically meaningful.

\subsection{An energy for trajectories recovery}

Let $T \in \N^*$ be the number of time samples, further defining
$(t_i)_{1\leq i\leq T}$ the time slices ranging from 0 to 1: $0 \leq
t_i \leq 1$, and $\hilb$ is an Hilbert space. The following
functional is the core component of the dynamic off-the-grid
framework \cite{Bredies2019, Bredies2019a} so far:

\begin{align}
  \argmin_{\sigma\in \mathcal{M}(\Gamma)} E(\sigma) =
  \argmin_{\sigma\in \mathcal{M}(\Gamma)} \sum_{i=1}^T \norm{A_i
  e_{t_i \sharp} \sigma - b_{t_i} }_\hilb + \int_{\Gamma} w(\gamma)
  \ud \sigma(\gamma),
  \label{eq:optim_sigma}
\end{align}

We ought to precise some of the quantities introduced here:

\begin{itemize}
  \item $t \mapsto b_t \in \hilb$ are the acquired data called the
    \emph{observation},
  \item $\Gamma = \left\{\gamma=(h,\xi), \quad h\in C([0,1],
      \R),\quad \xi : [0,1] \xrightarrow{} \xx, \quad \xi_{|h\neq 0}
    \text{ is continuous} \right\}$ is the space of trajectories with
    $h$ the amplitude weighting the mass and $\xi$ the curve,
  \item $e_t$ is the measurable map of evaluation at time $t$,
    defined by $e_t(\gamma) = \gamma(t)$ hence $e_{t \sharp} \sigma
    \in \mathcal{M}(\Omega),$
  \item $A_i : \mathcal{M}(\Omega) \xrightarrow{} \hilb $ a linear
    operator, pertaining to the physical context (in the case of
    super-resolution, the Gaussian convolution for example). The $i$
    index shows that the operator can incorporate some changes over
    time (such as the spread of the Gaussian convolution),
  \item $w : \Gamma \xrightarrow{} \R_+$ a weight function, usually
    $w(\gamma) = \int_0^T \alpha + \beta |\Dot{\gamma}|^2(t) \ud t$
    for $\alpha, \beta > 0$. Once integrated and for $\alpha=0$, it
    constitutes the Benamou-Brenier regularisation; it only caters
    for the balanced case, \ie\ conserved mass such as $\gamma=(1,\xi)$ .
\end{itemize}

\begin{remark}
  $\mathcal{M}(\Gamma)$ is well-defined as a \emph{complete separable
  metric space}, since $\Gamma$ is neither locally compact nor
  $\sigma$-compact. See \cite{Tovey2021} for more grounded arguments.
\end{remark}

This lies into the general framework, now set $\hilb=\R^n$ and:

\begin{align*}
  A_i : \rho \in \mathcal{M}(\Omega) \mapsto
  \left(\int_{\Omega}a_i^j(x)\ud \rho(x) \right)_{1\leq j \leq n} \in \hilb.
\end{align*}

Also we need to reduce ourselves to physically meaningful curves,
with some continuity assumption \cite[Definition 1.1.1]{Ambrosio2008}:

\begin{definition}[Absolutely continuous curves]
  Absolutely $p$-continuous curves in a metric space $(X,d)$ where
$p\in ]1,+\infty]$ are curves $\gamma : [0,1] \longrightarrow X$ such
that there exists $\zeta \in L^p([0,1], \R)$ and for all $s,t$ \st{}
$0\leq s\leq t\leq 1$:

\begin{equation}
d \left(\gamma(t),\gamma(s) \right) \leq \int^t_s \zeta (u)\ud u.
\label{eq:AC_def}
\end{equation}

The set of such curves is denoted $\AC^p([0,1],X)$, or simply
$\AC([0,1],X)$ if $p=1$.
\label{def:AC_curve}
\end{definition}

More properties on absolute continuous curves such as their metric
derivative, equal to the metric-norm of the classical derivative, are
expressed in Appendix \ref{app:riemann-recall}. It gives rigorous
definitions of these concepts for the Riemannian generalisation
involved in the following sections.
Now, we recall the main result \cite[Theorem 10]{Bredies2019} for the
Benamou-Brenier regularisation.

\begin{theorem}[Solutions of balanced recovery from]
The energy $E$ admits a minimiser $\rho^*$ which is a finite sum of
extreme points of the Benamou-Brenier unit ball, \ie\:
\begin{align*}
\exists a_i >0, \xi_i \in \AC^2([0,1],\xx), \quad \forall t \in
[0,1], \quad \rho^*(t) = \sum_{i=1}^T a_i \delta_{\xi_i(t)}
\end{align*}
\end{theorem}

The unbalanced case, \ie\ with varying amplitude, has been treated in
the literature through the Wasserstein-Fisher-Rao (WFR)
regularisation \cite{Bredies2022a,Tovey2021}. In the following, we
will introduce our Riemannian framework for the balanced case, but we
stress that our ideas hold even for WFR regularisation, which is more
suited to the biomedical case we strive to solve.

The seminal papers have proven the efficiency of this optimal
transport based method, as it has been shown to successfully recover
the paths of moving spikes. However, the numerical implementation
struggles to recover crossing paths: this is an expected behaviour,
as the algorithm cannot infer the correct paths from the dual
certificate (see \cite{Duval2014}). This ambiguity is inherent to the
considered energy: in the following, we will explore our proposed
workaround which exploits a lifting to the space of roto-translation
inspired by \cite{Chambolle2019, duits2018optimal}, while formulating
a new energy with a Riemannian regularisation tailored for paths untangling.

\section{Our proposed lifting: an insight in Riemannian and numerics}

\label{sec:theory}
\subsection{The orientation-position space}
As we mentioned earlier, the original curve reconstruction model
struggles to recover multiple crossing curves in 2D. Indeed, the
regularising term in \eqref{eq:optim_sigma} selects solutions with
minimal length in the 2D plane, thus preferring couples of
nearly-touching curves to crossing curves. This is all the more
unfortunate as this formulation does not fully address the problems
encountered in the biomedical applications: as we pointed out before,
the curves frequently crossed themselves on the acquisition, and the
method above cannot handle well this rather frequent case.
To prevent such issue, we propose an approach consisting in lifting
our setting $\xx$ in the space of position and orientation, adding an
orientation coordinate to our 2D variable.

Hence, we lift the problem in a suitable space to allow the
disentanglement, further inspired by the variational approach
proposed by \cite{Chambolle2019}. The orientation-position space,
also called the roto-translation space in the Lie group literature,
is understood as a homogeneous space upon which the \emph{Special
Euclidean group} $\SE(d) = \R^d \rtimes \SO(d)$ acts transitively and
faithfully. $\SO(d)$ the d-dimensional rotation group, \ie\ the group
of all rotations about the origin of $d$-dimensional Euclidean space
$\R^d$ under the operation of composition.

As mentioned earlier, $\xx \subset \R^d$ is then the spatial domain
where the acquisition is defined, and $\sphere^{d-1}$ thereby
parametrises the local orientation of the curve.

\begin{definition}
  The homogeneous space of orientation-position in $\R^d$ is:
  \begin{align*}
    \rototrans_d \eqdef \R^d \times \sphere^{d-1}.
  \end{align*}

  We fill also refer to it as the \emph{roto-translation} space.
\end{definition}

\begin{remark}
  The semidirect product, denoted by $\rtimes$, fundamentally differs
  from the direct product ($\times$) in that it allows for a
  non-commutative interaction between its constituent groups. Unlike
  a direct product where elements from one subgroup commute with
  elements from the other, in a semidirect product, one subgroup acts
  on the other. Consequently, the composition of elements, such as
  two roto-translations, is generally order-dependent. Stricly
  speaking, $\rototrans_d \eqdef \R^d \rtimes \sphere^{d-1}$ but
  since we do not compose elements of $\rototrans_d$, choosing the
  semidirect product or the direct product here makes no substantial difference.
\end{remark}

It is a $(2d-1)\text{-dimensional}$ manifold, in the following we
will consider examples in $d=2$ since our biomedical application is
set on 2--dimensional images and we feel that this case is more
pedagogical: however, note that our solution still applies to any $n$D setting.
The space $\rototrans_2$ is useful because it allows one to
differentiate objects with the same position but with different
orientations, which is a common feature in $2D$ vessel images, as the
projection in a $2D$ image of a tree-like structure gives rise to crossings.

If needed and for computational ease, we will identify $n =
(\cos(\theta), \sin(\theta))\in \mathbb{S}^1 \longleftrightarrow
\theta \in \R/\mathbb{Z} \longleftrightarrow R_{\theta} \in \SO(2)$,
since $\sphere^1$ is isomorphic to $\SO(2)$. As of now, $N \in \N^*$
control points\footnote{Control points are a set of fixed points used
  to define and shape parametric curves, most notably Bézier curves,
  B-splines, and NURBS. They do not usually lie on the curve itself
(except the endpoints), but they determine its form and geometry.}
will no longer belong to $\xx^N$ but rather in the manifold
$\left(\rototrans_d \right)^N$ to further incorporate orientation
information in the curve. The roto-translational curve space is then:

\begin{align*}
  \Upsilon \eqdef \left\{\gamma=(h,\xi), \quad h\in C([0,1],
    \R),\quad \xi : [0,1] \xrightarrow{} \rototrans_d,\quad \xi_{|h\neq
  0} \text{ is continuous} \right\}
\end{align*}

and the measure space $\mathcal M (\Upsilon )$ follows naturally.
Eventually, note that the original roto-translational space
$\rototrans_d$ is a Riemannian manifold, hence a vector space with
more structure thanks to the metric tensor. This Riemannian structure
will be leveraged to further implement the regularisation on curves
crossing, see Figure \ref{fig:roto-illustration}, but beforehand it
calls for a quick introduction.

\subsection{A relaxed metric on the orientation-position space}

We ought to begin this section with a recall on differential
geometry, in particular for Riemannian or flavour-liked manifold, as
we strive to make this paper self-contained. Note that not only these
notions are mandatory to define our metric/regularisation, but they
are also needed to perform the numerical implementation through
Riemannian optimisation. Be careful as in this section, $\gamma$
denotes a geodesic and not the couple $(h,\xi) \in \Gamma$ of
amplitude and path.

\subsubsection{Riemannian geometry basics}
\label{sec:riemann}

As in 1946, Élie Cartan\footnote{Pr. Élie Joseph Cartan (1869-1951)
  was a French mathematician known for his celebrated work on Lie
  groups, mathematical physics and differential geometry -- a field he
contributed to revitalise.} stated that \emph{'the general notion of
manifold is quite difficult to define precisely'} \cite{Cartan1946};
in this vein the notions exhibited below will be defined rather
informally, as we think the reader should focus more on the
underlying idea than the formal setting. However, we strongly advise
the interested reader, willing to learn more about Riemannian
geometry, to take a glance at \cite{Lee2019, Agrachev2020,
Gudmundsson2002, Schrodinger2009} for a rather more far-reaching presentation.

Let $\M$ be a smooth $d$-manifold, namely a space where one can lay
down everywhere local coordinates $(x^1, \dots, x^d)$
\ie\ homeomorphism mapping open set of $\M$ to an open set of $\R^d$.
It is a generalisation of the notion of smooth surface of $\R^d$,
that one can bend, twist, and curve smoothly in any direction without
any creases or sharp corners; bar the manifold does not rely on the
ambient space for its definition. A smooth manifold is then a space
that locally resembles a Euclidean space.
On each point $p \in \M$ one can define the tangent space $\Tang_p
\M$: it can be seen as a linearisation of the manifold at $p$. A
tangent space is a vector space, enjoying an ordered basis.
The union of all point of the manifold and their associated tangent
space is called the tangent bundle $\TM \eqdef \left\{(p,v) \,\vert\,
p \in \M, v \in \Tang_p\M \right\}$.

The manifold can be equipped with a metric $g$ \ie\ a symmetric
covariant tensor field, also denoted $g_{ij}$ in tensor notation,
defined on the tangent bundle. The restriction at $p \in\M$ is
denoted by $g_p: \Tang_p\M \times \Tang_p\M \to \R$, it is an inner
product that maps two vectors lying on the tangent space of $p$ onto
a real number. Loosely speaking, a Riemannian manifold $(M,g)$ is a
smooth manifold with a smoothly varying inner product $g_p$ on the
associated tangent spaces. We further denote $\norm{\cdot}_g \eqdef
\sqrt{g(\cdot, \cdot)}$ the metric norm\footnote{The $p \in \M$ index
  are dropped for clarity and conciness, but it is important to
remember that the metric norm depends on the point $p$ where it is evaluated.}.

The metric admits a contravariant inverse tensor reading $g^{-1}$, or
$g^{ij}$. It is uniquely defined by $\delta^{i}_k = g_{jk} g^{ij}$
where $\delta_{ij}$ is the Kronecker symbol. A core component of
Riemannian geometry is the geodesic, the generalisation in the
Riemannian manifold $(\M, g)$ context of straight lines of the
Euclidean setting.

\begin{definition}[Geodesic]
  \label{def:geodesic}
  In the sense of the calculus of variations, a \emph{geodesic} is a
  continuously differentiable $\gamma : [0,1] \to \M$ such that it is
  a minimum of the energy:

  \begin{align}
    E_\mathrm{c}(\gamma) \eqdef \dfrac{1}{2} \int_0^1 g_{\gamma(t)}
    \left(\gamma'(t), \gamma'(t) \right) \ud t.
    \label{eq:kinetic-geodesic}
  \end{align}

  A geodesic can be thought of as a free particle $\gamma(t)$ such
  that its acceleration has no component in the direction of the
  surface, hence at each point $p = \gamma(t)$ of the curve it is
  perpendicular to the tangent plane $\Tang_p \M$ of the
  manifold\footnote{The motion of the particle is then of constant
  speed and solely defined by the bending of the surface.}.
\end{definition}

A Riemannian manifold can then be endowed with the \emph{geodesic metric}, then

\begin{align}
  \forall (p, q)\in \M^2, \quad \dg(p,q) \eqdef
  \inf_{\substack{\gamma \in \AC([0,1],X) \\ \gamma(0) =p, \gamma(1)
  =q}} \int_0^1 \norm{\gamma(t)}_g^2 \ud t.
  \label{eq:geodesic-distance}
\end{align}

Consider a smooth map on the manifold $f:\M \to \R$, its Riemannian
gradient at $x$ lies in the tangent space $\mathrm{grad} f \in
\Tang_x\M$. For $\M \subset \R^d$, it is related to the Euclidean
gradient $\grad_\R f : \R^d \to \R$, the classical gradient defined
in the ambient space, by:

\begin{align*}
  \grad f = g^{-1} \grad_\R f, \quad \mbox{also written } \grad f_i =
  g^{ij} \grad_\R f_j.
\end{align*}

Obviously, a vector $v$ of the tangent plane $\Tang_x\M$ has no
reason to belong to the manifold, and an iterate of the gradient
descent might then escape the manifold. A tool from the differential
geometry can then be leveraged:

\begin{definition}[Exponential map]
  \label{def:exp-map}
  Let $v \in \Tang_p \M$ a tangent vector of the manifold at point
  $p$. Then, there exists a unique geodesic $\gamma_v : [0,1] \to \M$
  such that $\gamma(0) = p$ and $\gamma'_v(0) = v$. The
  \emph{exponential map} is defined by:
  \begin{align*}
    \Exp_p(v) = \gamma_v(1).
  \end{align*}
  The application $\Exp_p$ acts as a local diffeomorphism from the
  tangent space $T_p \M$ to $\M$. When defined, its inverse $\Log_p:
  \M \to \Tang_p M$ is called the \emph{logarithmic map}.
\end{definition}

The optimisation in the Riemannian setting is a powerful tool
leveraging the structure offered by the manifold: as the solution
must lie in $M$, each iteration remaps the new iterate onto the
manifold with the exponential map. In contrast, Euclidean
optimisation may escape the manifold and might struggle to converge
to the solution, getting trapped into local minima outside the set,
in the ambient space (if defined).

In Riemannian geometry and in particular for optimisation, we may
need to move a vector defined on a tangent space $\Tang_x\M$ to
another tangent space $\Tang_y\M$. The new tangent space $\Tang_y\M$
enjoys an ordered basis, differing from those of $\Tang_x\M$: these
changes in the two coordinate systems are encapsulated in the
Christoffel symbols $\Gamma^k_{ij}$: it accounts for the change in
the $i$-th component caused by a change of the $j$-th component,
computed in the $k$-th component direction. Christoffel symbols
completely define the metric tensor, and conversely; it also comes
handy for second order Riemannian optimisation scheme \footnote{As
  the Riemannian Hessian of $f$ function of $(x_1, \dots, x_d)
  \in\M$ is defined by:

  \begin{align*}
    \Hess f_{ij} \eqdef \pdv{f}{x_i}{x_j} - \Gamma^k_{ij} \pdv{f}{x_k}.
  \end{align*}
}.
Moreover, it is useful for geodesics computation as it appears in
this fundamental ordinary equation:

\begin{theorem}[Geodesic equation]
  Given initial conditions $p \in \M$, $v \in \Tang_p \M$, a geodesic
  $\gamma: [0,1] \to \M$ is a  solution of:

  \begin{align}
    \dv[2]{\gamma^k}{t} + \Gamma_{ij}^k \dv{\gamma^i}{t}  \dv{\gamma^j}{t} = 0,
    \label{eq:geodesic}
  \end{align}

  with $\gamma(0) = p$ and $\gamma'(0) = v$. This is simply the
  Euler-Lagrange equations of the action in equation
  \eqref{eq:kinetic-geodesic}, expressed in local coordinates:
  $\gamma^i \eqdef x^i \circ \gamma (t)$.
\end{theorem}

Eventually, for the optimisation in a Riemannian setting, we advise
the reader to look up the Appendix \ref{app:riemannian-optim}.

These notions are rather general but compulsory as we work with the
$\rototrans^d$ manifold, hence we need to perform a optimisation
exploiting the underlying Riemannian structure. We can now leverage
this knowledge with our proposed lift on the rototranslation space
$\rototrans_d$.

\subsubsection{The Reeds-Shepp metric}
Multiple models have taken advantage of the space of roto-translation
$\rototrans_d$, among which applications we could quote computing
geodesics taking into account orientation features and penalising
curvature. Usually used to model vehicles with forward and reverse
gears or oriented objects in $\R^d,$ the Reeds-Shepp metric
\cite{Reeds1990} has numerous applications in various context. It is
a sub-Riemannian\footnote{A sub-Riemannian manifold is a Riemannian
  manifold coupled with a constraint on admissible directions of
movements.} metric that prevents curves from being planar and
penalises changes of direction (meaning penalising curvature of the
trajectory). We will use a relaxed Riemannian version $\RSe$ of the
Reeds-Shepp metric $\RS$ here that only penalises velocities
orthogonal to. This model continuously converges to the standard
sub-Riemannian Reeds-Shepp metric  \cite{duits2018optimal} as the
relaxation parameter $\varepsilon$ yields $0$.

There is several examples of the roto-translational space and
Reeds-Shepp metric in medical images processing, such as in
\cite{Bekkers_2014} which applies filters on lifted images to help
track blood vessels in vascular imaging and in \cite{bekkers_2015}
for the computation of geodesics in $\SE(2)$. The following
definition of $\RSe$ is peculiar to the biomedical application.

\begin{definition}[Relaxed Reeds-Shepp metric]
  The relaxed Reeds-Shepp metric tensor field $g$
  \cite{duits2018optimal} can be written as a Euclidean norm at each
  point in $\rototrans_2$. Let $0 < \varepsilon \leq 1$ be the
  relaxation parameter, $\xi > 0$ a scaling parameter, consider
  $(x,\theta) \in \rototrans_2$ while $(\Dot{x}, \Dot{\theta}) \in
  \Tang_{(x,\theta)} \rototrans_2$ lies in the tangent plane, then:

  \begin{align*}
    g_{(x, \theta)} \left((\Dot{x}, \Dot{\theta}),(\Dot{x},
    \Dot{\theta})\right) &\eqdef |\Dot{x}\cdot e_{\theta}|^2 +
    \frac{1}{\varepsilon^2}|\Dot{x}\wedge e_{\theta}|^2 +
    \xi^2|\Dot{\theta}|^2 \\
    & = \left<R_{\theta} (\Dot{x}, \Dot{\theta}) ,
    \mathrm{diag}\left(1, \frac{1}{\varepsilon^2}, \xi^2 \right)
    R_{\theta} (\Dot{x}, \Dot{\theta}) \right>_{\SO(2)},
  \end{align*}

  where $e_{\theta}$ is the unit vector in the direction defined by
  the angle $\theta$, $R_\theta$ is the rotation
  matrix\footnote{$R_{\theta} \eqdef
    \begin{pmatrix}
      \cos \theta & -\sin \theta\\[6pt]
      \sin \theta & \cos \theta
    \end{pmatrix} \in \SO(2)$, and $R_\theta (\dot x, \dot \theta) = R_{\theta}
    \begin{pmatrix}\dot x\\ \dot \theta
  \end{pmatrix}$} and the bracket is the classical Euclidean inner product.
\end{definition}

The idea of this relaxation boils down to the increasing penalisation
of non-planarity of the lifted path as $\varepsilon$ nears $0$.
The relaxed version of Reeds-Shepp accurately approaches the
sub-Riemannian case with infinite cost for non-planar curves. The
$\abs{\Dot{\theta}}^2$ term penalises steering the prescribed
orientation of the curve, and acts as a penalisation of the local
curvature of the curve.
As we have defined a regularisation function embodied in the
Reeds-Shepp metric, the problem is further written down as the
optimisation of the following energy:

\begin{align}
  \argmin_{\sigma\in \mathcal{M}(\Upsilon)} T_{\beta,
  \varepsilon}(\sigma) = \argmin_{\sigma\in \mathcal{M}(\Upsilon)}
  \sum_{i=1}^T \norm{A_i e_{t_i \sharp} \sigma - b_{t_i} }_\hilb +
  \beta \int_{\Gamma} \int_0^1 \norm{\gamma' (t)}_{g} \ud t \ud \sigma(\gamma).
  \label{eq:unravelling-blasso}
\end{align}

Theoretical guarantees proven by the original authors still hold, see
appendix \ref{app:extreme} for a more rigorous proof that the paths
extreme points are conserved with this energy \eqref{eq:unravelling-blasso}.
Before illustrating our proposed energy on some numerical examples,
one can genuinely raise some concern on the theoretical soundness of
discretising the measure problem, as in practice the UFW presented in
section \ref{sec:numerics} recovers discretised curves, since the
computers cannot embed a computation with infinite precision. Do we
need an infinity of control points to correctly describe the solution
of the variational problem?
In the following, we try to fully address this issue by considering
multiple discretisations of curve and by proving $\Gamma$-convergence results.

\section{\texorpdfstring{$\Gamma$}{}-convergence of several
discretisation in the Benamou-Brenier case}

This section is motivated by numerical soundness, but also covers a
point raised by  \cite[Lemma 3.3]{Tovey2021}. Indeed, for

\[
  W : \gamma \mapsto \Tilde \gamma \in \argmin \left\{ w(\chi) \, |
  \, \chi \in \Gamma, 0 \leq i \leq T, \chi(t_i) = \gamma(t_i)  \right\}
\]

one has $E \left(W_\sharp \gamma \right) \leq E(\gamma)$, with
$w(\gamma) = \alpha + \beta \int_0^1 \norm{\gamma'(t)}_g^2 \ud t$.
This implies the existence of a theoretically optimal discretisation
of the space: any candidate solution to the optimisation problem can
be replaced by a better one, as long as it preserves the position of
the curve at the timestamps where the data term is evaluated. In
other words, the optimal curves are piecewise 'geodesic' with respect
to the weight function $w$ between each pair of timestamps $t_i$.

We now turn to discretising the space of curves to obtain a good
approximation of the solution to problem \eqref{eq:optim_sigma}, and
\emph{a fortiori} the energy in equation
\eqref{eq:unravelling-blasso} with the Reeds-Shepp metric. In this
section, we present a formulation for approximating absolutely
continuous curves using families of curves defined by a finite number
of control points. Note that these results are expressed for any $\xx
\subset \R^d$ threfore $\M(\Gamma)$, hence holds for submanifold
$\rototrans_d$ and associated $\M(\Upsilon)$. We then consider a
generic $\Gamma$ with its metric tensor $g$ (\eg\ Reeds-Shepp).

\subsection{A suitable definition for convergence}

We need to qualify the convergence of the surrogate problems to the
continuous one.
For this point, we will rely on the following notion of functional convergence.

\begin{definition}[$\Gamma$-convergence]
  Let $X$ a first countable space, $F_n : X \to \R$ a sequence of
  functionals. We say that $F_n$ sequence of functions
  \emph{$\Gamma$-converges} to its \emph{$\Gamma$-limit} $F: X \to
  \R$ if the two conditions hold:
  \begin{itemize}
    \item (Liminf inequality) for every sequences $x_n \in X$ such
      that $x_n \cvet x$:
      \begin{align*}
        F(x) \leq \liminf_{n\to\infty} F_n(x_n).
      \end{align*}
    \item (Limsup inequality) there exist a sequence $x_n \in X$ such
      that $x_n \cvet x$ and :
      \begin{align*}
        F(x) \geq \limsup_{n\to+\infty} F_n(x_n).
      \end{align*}
  \end{itemize}
\end{definition}

When defined, the limit is denoted $\Gamma\mathrm{-lim}_{n\to
+\infty}$ or more concisely $\Gamma\mathrm{-lim}_{n}$, hence with the
latter notations $F = \Gamma\mathrm{-lim}_{n} F_n$.
The interested reader can take a deeper dive in \cite{Maso2012,
Braides2002} to learn more. This approach is common in calculus of
variations, as it allows difficult—often non-convex—problems to be
approximated by simpler ones that are more tractable. Among other
properties, this functional convergence ensures \cite[Remark
2.11]{Braides2002} the \wast\ convergence of the minimisers (up to
subsequences) of the discretised problems $F_n$ towards the
minimisers of the continuous one $F$.

\subsection{Discretisation on the space of curves}

We will now consider how to discretise the space of curves in order
to find a good approximation of the solution of the problem
\eqref{eq:optim_sigma}. In this section, we provide a formulation for
the approximation of absolutely continuous curves using families of
curves defined by finite-dimensional control points.

Let $P^n$ be a map from the space of control points $\R^{d\times
k_n}$, for $k_n \in \N$ the number of control points, and $P_n =
P^n(\R^{d\times k_n})$ the $n$-th order discretisation of the space
of curves. Let $\mu\in \mathcal{M}(\R^{d\times k_n}),$ then we have
$P^n_{\sharp}\mu \in \mathcal{M}(\Gamma).$
The following result ensures that, with a few assumptions, that the
sequence of approximation spaces $P_n$ are good approximations to
find the minimisers of the energy in equation \eqref{eq:optim_sigma}.
Let $\chi_{P_n}(\gamma) =0$ if $\gamma \in P_n$ and  $+\infty$ otherwise.

\begin{theorem}
  Let $P_n\subset \Gamma$ be a sequence of subsets of $\Gamma$, such
  that $\bigcup_{n}P_n$ is dense in $\Gamma.$ For every $\gamma\in
  \Gamma$, let there be a sequence of measurable maps $S_n : \gamma
  \in \Gamma \mapsto \R^{d\times k_n}$ such that $P^n(S_n(\gamma))
  \longrightarrow \gamma$ for the uniform convergence. Suppose
  $w(P^n(S_n(\gamma)))\leq w(\gamma).$ Then the energy $E$ in
  equation \eqref{eq:optim_sigma} constrained to $P_n$
  $\Gamma$-converges to $E$, \ie\ $\Gamma\text{-}\mathrm{lim}_n(
  E+\chi_{P_n}) = E$.
\end{theorem}

\begin{proof}
  Consider a sequence $(\sigma_n)_{n \geq 0}$ such that $\sigma_n
  \cvet \sigma$ in $\radoncourbe$. Then, since

  \[
    E(\sigma_n) \leq E_n(\sigma_n),
  \]

  by lower semi-continuity, we conclude that:

  \[
    E(\sigma) \leq \liminf_{n\to+\infty} E_n(\sigma_n).
  \]

  For the $\limsup$ inequality, choosing $\sigma_n = (P^n \circ
  S_n)_{\sharp} \sigma \xrightharpoonup{*} \sigma$, the hypothesis yields

  \[
    \int_\Gamma w(\gamma) \, \mathrm{d}\sigma_n \leq \int_\Gamma
    w(\gamma) \, \mathrm{d}\sigma,
  \]

  and thus:

  \[
    \limsup_{n\to+\infty} \int_\Gamma w(\gamma) \, \mathrm{d}\sigma_n
    \leq \int_\Gamma w(\gamma) \, \mathrm{d}\sigma.
  \]

  By continuity, we finally yield:
  \[
    \limsup_{n\to+\infty} E_n(\sigma_n) \leq E(\sigma).
  \]
\end{proof}

\begin{remark}
  We may notice that the former proof defines an approximation map on
  $P_n.$ Let $S_n : \gamma \mapsto S_n(\gamma)$ be a measurable map,
  then $\gamma^n = P^n(S_n(\gamma)),$ and we retrieve $\sigma_n =
  \left(P^n \circ S_n \right)_{\sharp} \sigma.$ If we have $S_n\circ
  P^n = \mathrm{Id}$ as in the case of polygonal lines, it is even a projection.
\end{remark}

The latter result requires no more than very general hypotheses and
can be extended to more general settings. For instance, we may
consider the time-continuous setting \cite{Bredies2019a} where the
data term is changed to $\int_0^1 \|A e_{t\sharp} \sigma -
b_t\|^2_\hilb \ud t$, although we loose the application of the
representer theorem \cite{Boyer2018, Bredies2019b} that guarantees
the shape of the minimiser. Moreover, the only hypothesis on the
integrand $w(\gamma)$ relies on its lower semicontinuity, so we could
now consider many different regularisers, such as the Reeds-Shepp
metric used earlier.
On the contrary to \cite{Laville2024}, one do not have to exhibit a
discrete set of solutions to get the convergence, only the inequality
for the limsup property.
We can then detail some classic curve discretisation, fulfilling the hypothesis.

\begin{definition}[Curve discretisation]
  \label{def:discretisation}
  Let $c \in \Gamma^{k_n}$ be the control points of the discretised curve:
  \begin{itemize}
    \item Polygonal lines, $k_n$ \st\ $k_n | k_{n+1}$ ($k_n$ divides
      $k_{n+1}$) hence  $P_n \subset P_{n+1}$, $\tilde{t}_k = \frac{k}{k_n}$,

      \[
        P^n(c)(t) = \frac{\tilde{t}_{k+1} - t}{\tilde{t}_{k+1} -
        \tilde{t}_k} c_k + \frac{t - \tilde{t}_k}{\tilde{t}_{k+1} -
        \tilde{t}_k} c_{k+1}, \quad \text{if } t \in [\tilde{t}_k,
          \tilde{t}_{k+1}[
          \]

        \item Bézier curves, $P^n(c)(t) = \sum_{k=0}^n t^k
          (1-t)^{n-k}\binom{n}{k}c_k$,
        \item Piecewise geodesic\footnote{see Definitions
          \ref{def:geodesic} and  \ref{def:exp-map}.}, for $k_n |
          k_{n+1}$, there exists $0 \leq k \leq n$ such that $t\in
          \left[\frac{k}{n}, \frac{k+1}{n} \right]$

          \[
            P^n(c)(t) = \Exp_{c_k}\left( k_n \left( t-\frac{k}{k_n}
            \right) \Log_{c_k}(c_{k+1}) \right)
          \]
      \end{itemize}
    \end{definition}

    \begin{claim}
      The latter operators verify $w(P^n(S_n(\gamma))) \leq w(\gamma)$.
    \end{claim}

    \begin{proof}
      For the sake of clarity, we exhibit the inequality for one
      curve $\gamma$, meaning $\sigma = \delta_\gamma$; the
      generalisation to multi-curve $\sigma = \sum_i \alpha_i
      \delta_{\gamma_i}$ is straightforward by the triangle
      inequality of the metric norm.
      Let us check that the three classical discretisations fulfill
      the property:

      \begin{itemize}
        \item Polygonal curves case: taking $(S_n(\gamma))_k =
          \gamma(\Tilde{t}_k),$
          \begin{align*}
            & \int_0^1\left\|P^n(S_n(\gamma))'(t)\right\|_g^2\ud t \\
            &= \int_0^1 \left\|\sum_{k=0}^{k_n-1}
            \mathbbm{1}_{[\Tilde{t}_k, \Tilde{t}_{k+1}]}(t)
            \left[\gamma(\Tilde{t}_{k+1})-\gamma(\Tilde{t}_k)
            \right]\right\|_g^2/(\Tilde{t}_{k+1}- \Tilde{t}_{k})^2 \ud t\\
            & = \sum_{k=0}^{k_n-1} (\Tilde{t}_{k+1}-
            \Tilde{t}_{k})\left\|\int^{\Tilde{t}_{k+1}}_{\Tilde{t}_{k}}
            \gamma'(s)\ud s \right\|_g^2 / (\Tilde{t}_{k+1}- \Tilde{t}_{k})^2 \\
            & \leq \sum_{k=0}^{k_n-1}
            \int^{\Tilde{t}_{k+1}}_{\Tilde{t}_{k}}
            \left\|\gamma'(s)\right\|_g^2 \ud s\\
            & = \int^1_0\left\|\gamma'(s)\right\|_g^2 \ud s.
          \end{align*}
        \item Bézier curves: taking $S_n(\gamma) \in \R^{d\times
          (n+1)}$, then $ (S_n(\gamma))_k = \gamma(k/n),$
          \begin{align*}
            &\int_0^1\left\|P^m(S_n(\gamma))'(t)\right\|_g^2\ud t \\
            & =
            \int_0^1\left\|\sum^{m-1}_0\binom{m-1}{k}t^k(1-t)^{m-1-k}(\frac{\gamma((k+1)/m)
            - \gamma(k/m)}{1/m})\right\|_g^2\ud t\\
            & \leq_\text{Jensen} \sum_{k=0}^{m-1}\binom{m-1}{k}
            \int_0^1 t^k(1-t)^{m-1-k}\left\|\frac{\gamma((k+1)/m) -
            \gamma(k/m)}{1/m}\right\|_g^2\ud t\\
            & = \sum_{k=0}^{m-1}\binom{m-1}{k} \int_0^1
            t^k(1-t)^{m-1-k}\left\|\frac{\int_{k/m}^{(k+1)/m}\gamma'(s)\ud
            s}{1/m}\right\|_g^2\ud t\\
            &\leq_\text{Jensen} \sum_{k=0}^{m-1}\binom{m-1}{k}
            \int_0^1 t^k(1-t)^{m-1-k}\left(
              \frac{\int_{k/m}^{(k+1)/m}\left\|\gamma'(s)\right\|_g^2\ud
            s}{1/m} \right)\ud t\\
            & = \sum_{k=0}^{m-1} (1/m) \,
            \frac{\int_{k/m}^{(k+1)/m}\left\|\gamma'(s)\right\|_g^2\ud
            s}{1/m} \\
            &= \int^1_0\left\|\gamma'(s)\right\|_g^2 \ud s.
          \end{align*}
          \\
        \item Piecewise geodesic curves, $\dg$ the geodesic distance
          defined in equation \eqref{eq:geodesic-distance}:
          \begin{align*}
            & \int_0^1\left\|P^n(S_n(\gamma))'(t)\right\|^2\ud t \\
            &= \int_0^1 \left\|\sum_{k=0}^{k_n-1}
            \mathbbm{1}_{[\Tilde{t}_k,
            \Tilde{t}_{k+1}]}(t)\Log_{\gamma(\Tilde{t}_k)}
            \left(\gamma(\Tilde{t}_{k+1}) \right)
            \right\|_g^2/(\Tilde{t}_{k+1} - \Tilde{t}_{k})^2 \ud t\\
            & = \sum_{k=0}^{k_n-1} (\Tilde{t}_{k+1} -
            \Tilde{t}_{k})\left\|\Log_{\gamma(\Tilde{t}_k)}
            \left(\gamma(\Tilde{t}_{k+1})\right) \right\|_g^2 /
            (\Tilde{t}_{k+1} - \Tilde{t}_{k})^2 \\
            & = \sum_{k=0}^{k_n-1} \dg(\gamma(\Tilde{t}_k),
            \gamma(\Tilde{t}_{k+1}))^2 / (\Tilde{t}_{k+1} - \Tilde{t}_{k}) \\
            & \leq \sum_{k=0}^{k_n-1} \left(
              \int^{\Tilde{t}_{k+1}}_{\Tilde{t}_k} \norm{\gamma'(s)}_g
            \ud s \right)^2 / (\Tilde{t}_{k+1} - \Tilde{t}_{k}) \\
            & \leq_\text{Jensen} \sum_{k=0}^{k_n-1}
            \int^{\Tilde{t}_{k+1}}_{\Tilde{t}_k}
            \left\|\gamma'(s)\right\|_g^2 \ud s\\
            & = \int^1_0\left\|\gamma'(s)\right\|_g^2 \ud s.
          \end{align*}
      \end{itemize}
    \end{proof}

    \begin{remark}
      As we stated earlier, our result is related to Lemma 3.3 and
      3.4 in \cite{Tovey2021}, recalling that the minimiser
      $\sigma^*$ should be supported on the minimisers of $w$
      interpolating points prescribed at the timestamps appearing in
      \eqref{eq:optim_sigma}. The corresponding discretisations are
      the \emph{polygonal lines} for the Euclidean metric and more
      generally the \emph{piecewise geodesic curves} in the case
      where $w(\gamma) = \alpha + \beta\int_0^1
      \norm{{\gamma'}}(t)^2\ud t.$ However, it is still worth looking
      for other discretisations as it may be numerically helpful
      during the optimisation, such as a Bézier formulated with
      control points lying on the manifold..
    \end{remark}

\section{An algorithm and numerical illustrations of path untangling}

\label{sec:numerics}

\subsection{An Unravelling Frank-Wolfe algorithm}

Now that the lift has been leveraged to define an energy suited for
the untangling problem, one need to numerically implement the
reconstruction\footnote{See the associated repositories in
  \url{https://gitlab.inria.fr/blaville/dynamic-off-the-grid} and
\url{https://github.com/TheoBertrand-Dauphine/dynamic-off-the-grid/}}.
This is clearly non-trivial, as one needs to compute the minimisation
of a function on a infinite dimensional space, lacking the Hilbert
structure used in the proximal algorithms.
In the off-the-grid community, there are numerous algorithms
\cite{Castro2015,Chizat2019,Bredies2012,Denoyelle2018} to tackle the
numerical implementation of these off-the-grid methods, despite the
complexity of the task. The Frank-Wolfe algorithm \cite{Frank1956} is
one of them, this greedy algorithm works as a linearisation of the
objective function by its gradient, hence requiring only directional
derivatives to perform the optimisation.
It was refined for gridless methods in
\cite{Bredies2012,Denoyelle2018}, reaching the \emph{Sliding
Frank-Wolfe} algorithm which enjoys nice theoretical guarantees and
good   reconstruction in practice. The dynamic off-the-grid
implementation has been achieved in \cite{Bredies2022} and pursued in
\cite{Tovey2021} with some stochastic and dynamical programming improvements.

In this section, we introduce the \emph{Unravelling Frank-Wolfe}
algorithm, denoted UFW. It relies not only on the discretisations
quoted before, but it also exploits the Riemannian structure enjoyed
by the space $\rototrans_d$. As stated before, Riemannian
optimisation acts as a preconditioning of the gradient descent, hence
helping to avoid local minima and reducing the number of iterations.
Important Riemannian quantities for this implementation related to
the Reeds-Shepp metric are calculated in Appendix
\ref{app:reeds-shepp}. The Riemannian gradient descent exploited in
steps 7 and 8 is developed in Appendix \ref{app:riemannian-optim}.
The algorithm \ref{algo:ufw} shows the pseudocode for the UFW algorithm.

Since the Frank-Wolfe algorithm operates on a weakly-* compact set,
we will in practice reduce the set of curves to $C = \left\{ \gamma
\in \Upsilon, {w(\gamma)} \leq \norm{y}_\hilb / (2\lambda) \right\}$.

\begin{algorithm}[ht!]
  \SetAlgoLined
  \KwData{Acquisition $t \mapsto b_t \in \hilb$, number of iterations
    $K$, regularisation weight $\beta > 0$ and metric parameters $0
  \leq \varepsilon \leq 1$, $\xi > 0$.}

  Initialisation: $m^{[0]} = 0$, $N^{[0]} = 0$.

  \For{$k$, $0 \leq k \leq K$}{
    With $m^{[k]} = \sum_{i=0}^{N^{[k]}} a_i^{[k]}
    \delta_{{\curve_i}^{[k]}} $ such that $a_i^{[k]} \in \R$, $
    {\curve_i}^{[k]} \in C$, let
    \begin{align*}
      \eta^{[k]}(x) \eqdef \dfrac{1}{\alpha} \Phicurve^\ast
      (\Phicurve m^{[k]} - y).
    \end{align*}

    Find $\curve^\ast \in \Upsilon$ such that:
    \begin{align*}
      \gamma^\ast \in \argmin_{\gamma \in C}
      \crogamma{\eta^{[k]}}{\delta_\gamma}
    \end{align*} \label{algo:csfw-oracle}
    \eIf{$\norm{\eta^{[k]} \left(\gamma^\ast\right)}_\infty \leq 1$} {
    $m^{[k]}$ is a solution. \Return $m^{[k]}$.} {
      \label{algo:csfw-convex-step} Compute $m^{[k + 1/2]} =
      \sum_{i=0}^{N^{[k]}} a_i^{[k + 1/2]} \delta_{{\curve_i}^{[k]}}
      + a^{[k + 1/2]}_{N^{[k]}+1} \delta_{\gamma^*} $ such that:
      \begin{align*}
        a_i^{[k + 1/2]} \in & \argmin_{a \in \R^{N^{[k]}+1}}
        T_{\beta, \varepsilon} \left( \sum_{i=0}^{N^{[k]}} a_i
          \delta_{{\curve_i}^{[k]}} + a_{N^{[k]}+1} \delta_{\curve^*}
        \right) \quad
      \end{align*}

      \label{algo:csfw-non-convex-step} Compute $m^{[k+1]} =
      \sum_{i=0}^{N^{[k]} + 1} a_i^{[k+1]}
      \delta_{{\curve_i}^{[k+1]}}$, output of the optimisation
      initialised with $m^{[k + 1/2]}$:
      \begin{align*}
        \left(a^{[k+1]}, \curve^{[k+1]} \right) \in
        \argmin_{(a,\curve) \in \R^{N^{[k]+1}} \times
        \Upsilon^{N^{[k]+1}} } T_{\beta, \varepsilon} \left(
        \sum_{i=0}^{N^{[k]}+1} a_i \delta_{{\curve_i}}   \right).
      \end{align*}

      Prune the low amplitude atoms, set $N^{[k]}$ to the number of
      remaining atoms.
    }
  }
  \KwResult{Discrete measure $m^{[k]}$ where $k$ is the stopping iteration.}

  \caption{Unravelling Frank-Wolfe.}
  \label{algo:ufw}
\end{algorithm}

The optimisation in step 4 is the support estimation of the new
curve. The numerical certificate $\eta^{[k]}$ is defined thanks to
the extremality conditions \cite{Duval2014, Denoyelle2018} that
relate the optimal measure $\delta_\gamma$ to a continuous function
$\eta$. The numerical certificate $\eta^{[k]}$ amounts in practice to
the residual observation, after removing the contributions of the $k$
already estimated curves. A convenient initialisation is the one
minimising the square norm between each frame of $A_i e_{t_i \sharp}
\delta_\gamma$ and $ \eta^{[k]}$.
We also tested a multistart strategy for this step, applying the idea
of \cite{Bredies2022}.

Riemannian gradient descents in steps 7 and 8 are performed using
PyTorch and its automatic‐differentiation framework. To solve the
geodesic equation \eqref{eq:geodesic} for defining piecewise geodesic
curves and also to perform a Riemannian gradient descent, we rely on
numerical integration. In particular, we use the \texttt{torchdiffeq}
package\footnote{\url{https://github.com/rtqichen/torchdiffeq}},
which implements multiple integration schemes with autograd support,
such as the Dormand–Prince method, an adaptive‐step fifth‐order
Runge–Kutta method. Christoffel symbols $\Gamma^k_{ij}$ involved in
the geodesic equation are computed in Appendix \ref{app:reeds-shepp}.

\begin{remark}
  Obviously, piecewise geodesic discretisation implies constraint not
  only on $x \in \rototrans$ points of the geodesic but also for its
  direction/covariant derivative $v \in \Tang_x \rototrans$ at each
  discretisation point, by definition of the Exponential map. Hence,
  the optimisation must be performed on these $(x,v)$ which calls for
  some adaptation with a metric compatible with the tangent bundle,
  the \emph{Sasaki metric}. Its definition and the modified energy
  are detailed in Appendix \ref{app:riemannian-optim}.
\end{remark}

In the Euclidean setting, we perform gradient descent using the Adam
optimizer with a learning rate of \(10^{-2}\). Bézier curves are
generated via a naive implementation of de Casteljau’s algorithm
since our goal is not solely on performance, but since the timesteps
are fixed, the generation could be condensed in a pseudo-Vandermonde
matrix product, more lightweight and easier to compute than the plain
recursive algorithm.

Now, we can apply this UFW  algorithm to simulated data, and verify
that it is able to retrieve crossing curves from the blurred and
noisy observation, as shown in Figure \ref{fig:first_phantom}.

\begin{figure}[ht!]
  \centering
  \begin{subfigure}[b]{0.3\textwidth}
    \centering
    \includegraphics[width=\textwidth]{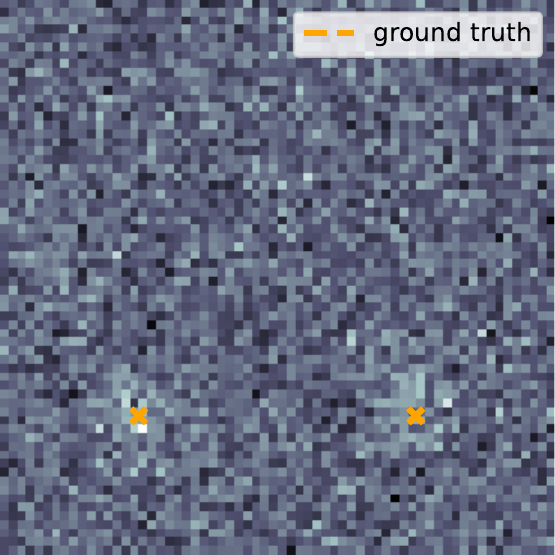}
  \end{subfigure}
  \hfill
  \begin{subfigure}[b]{0.3\textwidth}
    \centering
    \includegraphics[width=\textwidth]{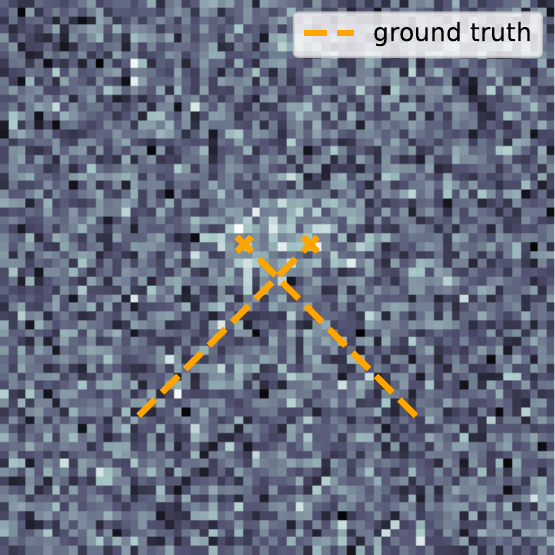}
  \end{subfigure}
  \hfill
  \begin{subfigure}[b]{0.3\textwidth}
    \centering
    \includegraphics[width=\textwidth]{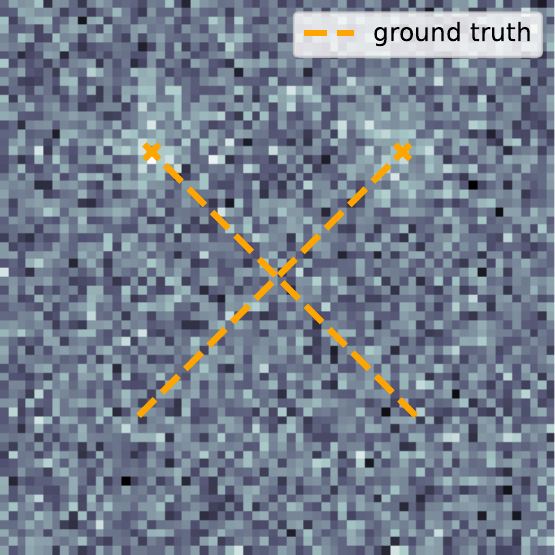}
  \end{subfigure}
  \caption{First phantom for $T=21$ frames and two spikes crossing
    themselves, 60\% of white Gaussian noise altering each frame (here
  at first, half of stack and last frame).}
  \label{fig:first_phantom}
\end{figure}

\subsection{Balanced optimal transport}

We consider two phantoms from \cite{Bredies2019a}, the former
consists in two point measures crossing, while the latter is a bit
more tedious with three Dirac measures moving nearby. The data term
is defined as a set of linear forms chosen as the integral of test
functions against the temporal measure $\rho_t$ at time $t.$ This
formulation covers the cases defined in \cite{Bredies2019b} where
Fourier coefficients are taken. We chose a classic test function for
these synthetic applications by setting the test functions as a
Gaussian kernel centred on vertices of a $N\times N$ grid covering
$[-1,1]\times [-1,1]$, defined likewise \cite{Denoyelle2018}. From an
application perspective, this may be seen as a discrete convolution
of the measure with the \emph{Point Spread Function} of the imaging system.
The Gaussian spread is set to $\sigma = 0.01$ on a $64 \times 64$ domain.

The first phantom is generated with $T=21$ time samples, and will be
tested on the worst-case scenario, namely 60\
noise altering each image: see Figure \ref{fig:first_phantom}. The
regularisation parameter $\beta$ (and so forth $\varepsilon$ and
$\xi$) has to be tuned properly.

We present in the figure \ref{fig:Euclidean-vs-RS} how different
choices reach opposite output, hence the proper choice of these
parameters is the crux for the untangling of curves.

\begin{figure}[ht!]
  \centering
  \begin{subfigure}{.5\textwidth}
    \centering
    \includegraphics[width=.835\linewidth]{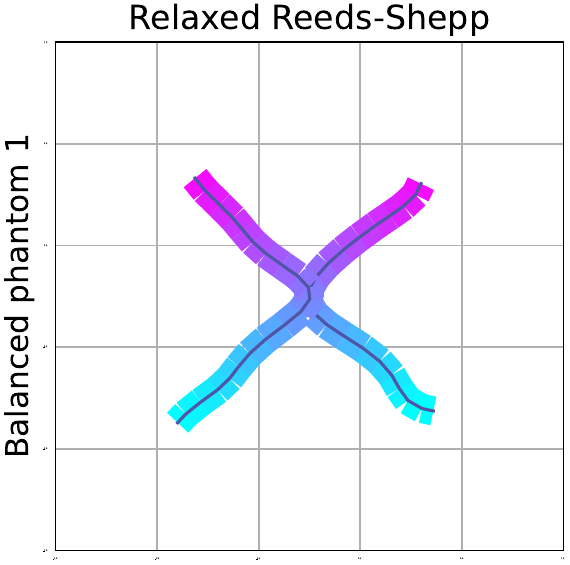}
  \end{subfigure}%
  \begin{subfigure}{.5\textwidth}
    \centering
    \includegraphics[width=.95\linewidth]{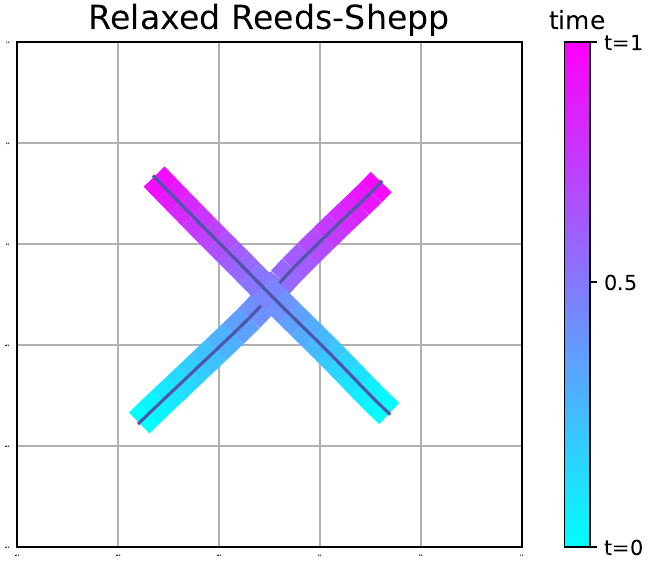}
  \end{subfigure}
  \caption{Reconstruction of the first phantom of \cite{Bredies2019a,
    Tovey2021} consisting of a stack of 21 images altered by 60\% of
    Gaussian noise. Each curve is a Bézier curve parametrised by control
    points living in $\rototrans_2$. Left: crossing is not recovered,
    with  $\beta = \num{e-3}$, $\varepsilon = \num{0.5}$ and $\xi = 100$.
    Right: untangling of curves, parameters are tuned to $\beta =
  \num{e-3}$, $\varepsilon = \num{0.05}$ and $\xi = 1$.}
  \label{fig:Euclidean-vs-RS}
\end{figure}

In layman terms, we could give the following interpretation:

\begin{itemize}
  \item $\beta$ controls the regularisation of our Benamou-Brenier
    flavoured regularisation,
  \item $\varepsilon$ enforces the planarity of the curve,
  \item $\xi$ penalises the local curvature.
\end{itemize}

Our method is then compared with state-of-the-art results from
\cite{Bredies2019a, Tovey2021} in figure \ref{fig:RS-vs-the-world}.
Our algorithm successfully untangles the trajectory, even with the
high level of noise. The computation time is approximatively 6
seconds on a Xenon CPU E5-2687W for polygonal, 50 seconds for Bézier
and up to 30 minutes for piecise geodesic.

\begin{figure}[ht!]
  \centering
  \begin{subfigure}{.5\textwidth}
    \centering
    \includegraphics[width=.84\linewidth]{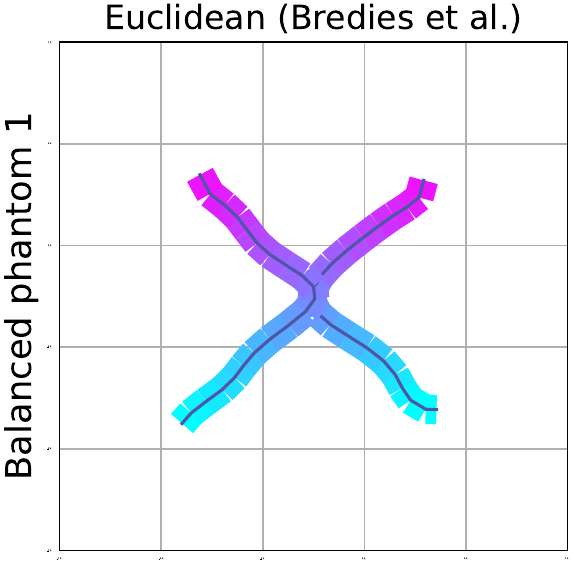}
  \end{subfigure}%
  \begin{subfigure}{.5\textwidth}
    \centering
    \includegraphics[width=.95\linewidth]{fig/2_dirac_crossing_RS.pdf}
  \end{subfigure}%
  \caption{Reconstruction of the first phantom of \cite{Bredies2019a,
    Tovey2021} consisting of a stack of 21 images altered by 60\% of
    Gaussian noise. Left: result from \cite{Bredies2019a}, quite similar
    to \cite{Tovey2021} up to the speed of convergence, the crossing
    cannot be recovered. Right: untangling of curves with our method,
    parameters are tuned to $\beta = \num{e-5}$, $\varepsilon =
  \num{0.05}$ and $\xi = 1$.}
  \label{fig:RS-vs-the-world}
\end{figure}

The choice of discretisation significantly affects the resulting
curves. Bézier curve discretisation facilitates the identification of
crossing curves, whereas polygonal discretisations may not. Bézier
curves are generally preferred as they are smooth and  their
non-locality allows for a good expressivity while they are
lightweight to compute. Otherwise, piecewise geodesic discretisation
bears a higher computation load (by to 2 orders of magnitude) as one
needs to solve at each iteration the geodesic equation
\eqref{eq:geodesic}, with the Christoffel symbols in Appendix
\ref{app:reeds-shepp}, but also yields good results often similar to
Bézier ones. In the following, we present only the results on Bézier
curves as it shows the best performance, but the reader should feel
free to examine the repositories for more information.

The second phantom now consists of $T=51$ slices, where one needs to
recover three paths. The results are plotted in figure
\ref{fig:RS-vs-the-world-bis}. The computation time is approximately
31 seconds on the Xenon, 2 minutes for Bézier and up to one hour for
piecewise geodesic.

\begin{figure}[ht!]
  \centering
  \begin{subfigure}{.5\textwidth}
    \centering
    \includegraphics[width=.84\linewidth]{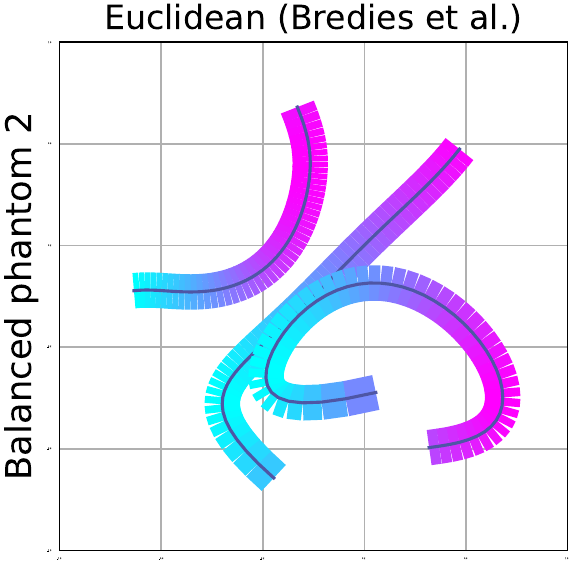}
  \end{subfigure}%
  \begin{subfigure}{.5\textwidth}
    \centering
    \includegraphics[width=.95\linewidth]{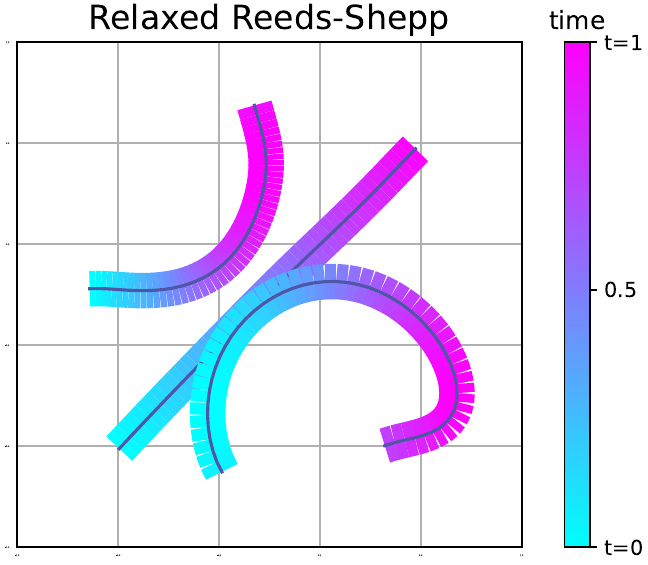}
  \end{subfigure}%
  \caption{Reconstruction of the second phantom of \cite{Bredies2019a,
    Tovey2021} consisting of a stack of 51 images altered by 60\% of
    Gaussian noise. The ground-truth spikes are not blending in each
    other, but rather take the same path at different times. Left: result
    from \cite{Bredies2019a}, quite similar to \cite{Tovey2021} up to the
    speed of convergence. Right: Reeds-Shepp regularisation, parameters
  are tuned to $\beta = \num{e-3}$, $\varepsilon = \num{0.05}$ and $\xi = 1$.}
  \label{fig:RS-vs-the-world-bis}
\end{figure}

\subsection{Unbalanced optimal transport}

Some further developments in the literature \cite{Tovey2021}, among
other improvements such as the computation time, proposed to tackle
the unbalanced optimal transport case: this encompasses any moving
Dirac with varying amplitude over time. This addition is a rather
important feature in the biomedical applications, since the molecules
of interest can lose brightness over time. Since the mass may change,
the regulariser is no longer the Benamou-Brenier energy but the quite
similar Wasserstein-Fisher-Rao one, as proposed in \cite{Bredies2019,
Bredies2019a} from the seminal work of \cite{Chizat2015,
Kondratyev2015, Liero2015}. Hence, the optimisation performed
consists of an \emph{unbalanced} optimal transport.

We propose adding up this idea to our algorithm to handle more
general problems. In practice, the amplitude information is
implemented in a new dimension added to the control points, the
estimated amplitude will thereby be a Bézier curve. For instance, in
the case $d=2$, the control points bear 4 dimensions, the first ones
for spatial dimension $(x_1, x_2)$, the third amounting to the
orientation $\theta$ (enabling the curve crossing) and the fourth
covering the amplitude $a$ (evolution of the mass over time) values
of the curve.

We noticed in our experiments that the algorithm tends to share
unfairly the global mass over the estimated paths: there is a
\emph{'winner takes it all'} effect, where the algorithm lean in
favour of the first estimated curve (overestimating it) while giving
less mass to the second curve (underestimating it). As for the
original article \cite{Tovey2021}, it did not enforce any prior on
the amplitude part and boils down to a least square norm fitting. We
propose to add a small regularisation with parameter $\zeta \in \R^+$:

\begin{align*}
  \mbox{if } \curve = \left(x_1(t), x_2(t), \theta(t), a(t) \right),
  \quad R(\curve) = \zeta \norm{a}_2^2.
\end{align*}

In the practical setting, this regulariser is added to the energy and
the energy is optimised accordingly. The case $\zeta = 0$ is
precisely the classical Wasserstein-Fisher-Rao unbalanced optimal
transport from \cite{Tovey2021}. In the following, we present in
figure \ref{fig:unbalanced-RS-vs-the-world} the reconstruction
performed by our method, against the reconstruction in the
literature, which lacks by design the curve crossing. Our method
successfully captures the dynamic of both paths and amplitudes
components, and this latter observation generalises even with noise.
Note that in the bottom left corner, our algorithm struggles to
recover the 'loose ends' namely the position where the amplitude of
the moving Dirac measure vanishes: the algorithms may handle paths
discontinuities by zeroing out the amplitude, but a near-zero
amplitude on the path endpoint is hard to capture since there is no
connectivity prior to exploit.

\begin{figure}[ht!]
  \centering
  \begin{subfigure}{.5\textwidth}
    \centering
    \includegraphics[width=.84\linewidth]{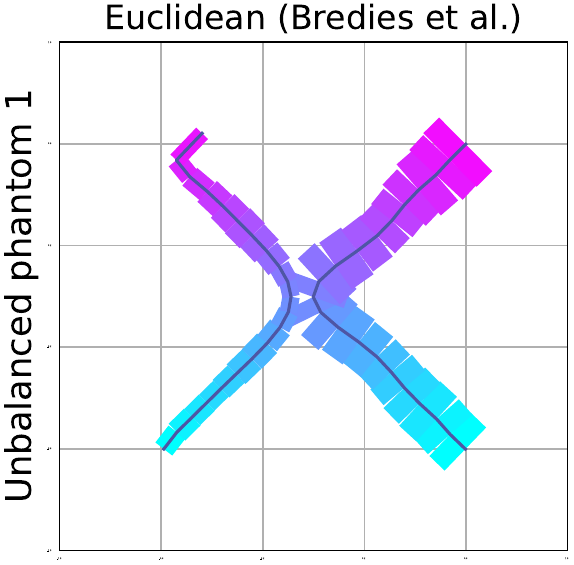}
  \end{subfigure}%
  \begin{subfigure}{.5\textwidth}
    \centering
    \includegraphics[width=.95\linewidth]{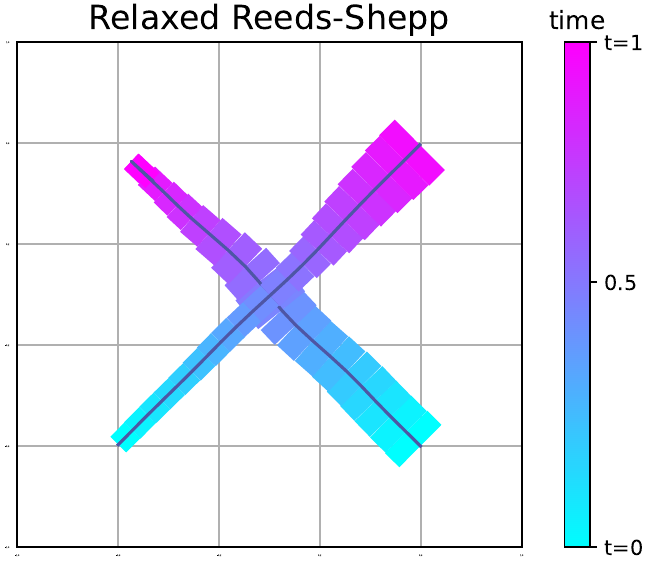}
  \end{subfigure}%
  \caption{Reconstruction of the first phantom of \cite{Bredies2019a,
    Tovey2021} consisting in a stack of 21 images with varying masses on
    the paths and 10 \% noise. Left: result for $\beta = \num{e-5}$
    without any lift, the crossing cannot be recovered. Right: untangling
    of curves with our method, parameters are tuned to $\beta =
  \num{6e-4}$, $\varepsilon = \num{0.05}$, $\xi = 1$ and $\zeta = 0.3 \beta$. }
  \label{fig:unbalanced-RS-vs-the-world}
\end{figure}

\begin{figure}[ht!]
  \centering
  \begin{subfigure}{.5\textwidth}
    \centering
    \includegraphics[width=.84\linewidth]{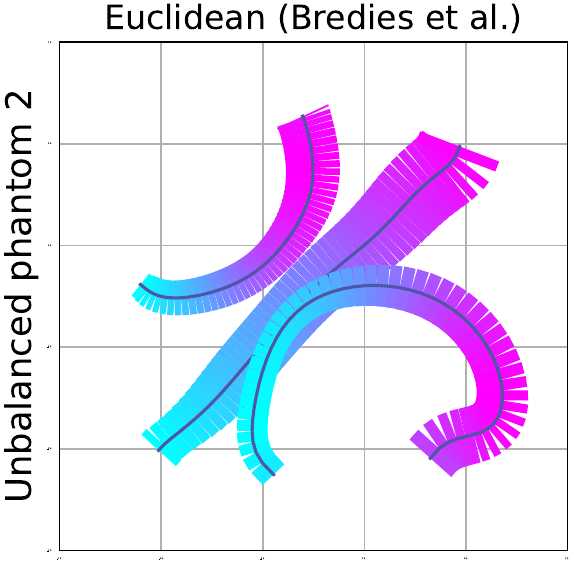}
  \end{subfigure}%
  \begin{subfigure}{.5\textwidth}
    \centering
    \includegraphics[width=.95\linewidth]{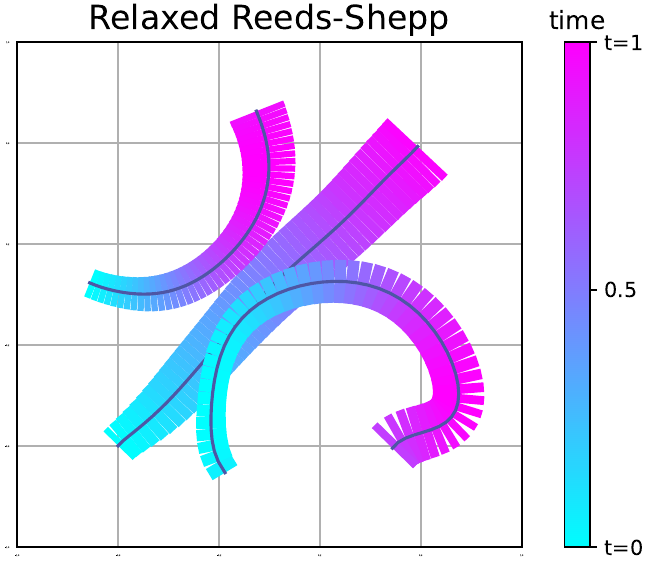}
  \end{subfigure}%
  \caption{Reconstruction of the second phantom of \cite{Bredies2019a,
    Tovey2021}, unbalanced case with 60 \% of Gaussian noise. Left:
    result for $\beta = \num{e-5}$. Right: Reeds-Shepp relaxation,
    parameters are tuned to $\beta = \num{6e-4}$, $\varepsilon =
  \num{0.05}$, $\xi = 1$ and $\zeta = 0.3 \beta$. }
  \label{fig:unbalanced-3-spikes-RS-vs-the-world}
\end{figure}

To sum this up, the Unravelling Frank-Wolfe algorithm is able to
recover crossing paths with various discretisations and possibly
varying masses over time. Bézier curves perform reasonably well
compared to piecewise geodesic discretisation, while being
lightweight to compute. In future work, it may be interesting to
consider Riemannian  Bézier curves compatible with the manifold
structure through a definition as the Fréchet mean on the control
points on the manifold rather than its current definition, which
still relies on the ambient space.

\section{Conclusion and outlook}

We successfully proposed a method to enable the reconstruction and
untangling of crossing moving point sources, through the use of a
lift with the roto-translational space $\rototrans_d$ and the
introduction of a proper metric to untangle the crossing path while
controlling the path curvature. This metric, inspired by biomedical
applications, has been used as a regularisation term. Then, we proved
$\Gamma$-convergence results for discretisation of our proposed
energy but not limited to, this may come handy for other off-the-grid
curve reconstruction.
Eventually, an Unravelling Frank-Wolfe algorithm has been introduced
to successfully tackle the augmented problem, through the leveraging
of Riemannian structure in both the inner gradient descents and the
metric in the objective energy. We also tested discretisations
considered in the latter section and proved the capabilities of
Bézier curves, compared to piecewise geodesic and polygonal ones.

Still, some difficulties arise naturally, and point out some
limitations such as curvature variation: untangling curves requires a
precise control on the curvature, which leads to a trade-off between
smoothness and entanglement. Fortunately, in real data, the curve
structure only bears a rather little curvature, which is encouraging
for precise untangling.

In future works, it would be interesting to consider the extension of
this method to off-the-grid divergence regularisation known to
promote curves \cite{Laville2023}, namely the reconstruction (and
untangling) of curves from one image rather than a stack. Both
problems, dynamic path and static curves, share many similarities and
could benefit from each other improvements.
Also, our proposed algorithm will be tested on real data images,
taken from Ultrasound localisation microscopy (ULM) \cite{Errico2015,
Couture2018} which aims to image this circulation, and relies on
multiple image frames containing microbubbles, flowing through the
blood vessels as an approximation to the real vasculature.
The overall algorithm could be delivered in a so-called
\emph{'off-the-shelf'} package, which may come handy for biologists
and applicative researchers with a GUI.

\section{Acknowledgments}

The authors thank Arthur Chavignon for the ULM dataset.
They are also deeply grateful to Gilles Aubert and Laure Blanc-Féraud
for their time and insightful feedback on the article.
BL would like to warmly thank TB for the shelter he provided in his
flat back in 2023, and for the pungent yet luscious homemade coffee
he got to taste then.

\clearpage
\bibliographystyle{siamplain}
\bibliography{biblio.bib}

\clearpage
\appendix
\section{Some precisions on geodesic and Riemannian setting}
\label{app:riemann-recall}

We begin by recommending that the reader review the relevant concepts
from differential geometry, as introduced in Section \ref{sec:riemann}.

\begin{definition}
  The \emph{metric derivative} \cite{Ambrosio2008} of an absolutely
  continuous curve $\gamma\in \AC([0,1],X)$ reads
  \begin{equation}
    \forall t \in [0,1],\quad |\Dot{\gamma}|(t) = \lim_{\eta \to 0}
    \frac{d(\gamma(t),\gamma(t+\eta))}{|\eta|},
  \end{equation}
  and it is the optimal function satisfying \eqref{eq:AC_def} as in
  definition~\ref{def:AC_curve}.
\end{definition}

We may now introduce the following lemma, useful in the Riemannian
context, see section \ref{sec:riemann} for details, .

\begin{lemma}
  If the distance $d$ is the Euclidean distance, and $\gamma\in
  \Gamma$ is differentiable, then the metric derivative corresponds
  to the norm of the classical derivative : $\forall t \in [0,1],
  \quad |\Dot{\gamma}|(t) = \|\gamma'(t)\|_g.$\\

  Similarly, if $d$ is the geodesic distance associated with the
  metric tensor $g$, then the metric derivative corresponds to the
  metric applied to the classical derivative : $\forall t \in [0,1],
  \quad |\Dot{\gamma}|(t) = \sqrt{g_{\gamma(t)}(\gamma'(t),\gamma'(t))}.$

\end{lemma}

\begin{proof}

  The Euclidean case comes directly from the definition of the
  derivative, and the Riemannian case is proven via:

  \begin{align*}
    &\frac{d(\gamma(t), \gamma(t+\eta))}{|\eta|} \\
    &= \frac{\sqrt{g_{\gamma(t)}
    \left(\Log_{\gamma(t)}(\gamma(t+\eta)),\Log_{\gamma(t)}(\gamma(t+\eta))\right)}}{|\eta|}
    \\
    &= \sqrt{g_{\gamma(t)}
      \left(\frac{\Log_{\gamma(t)}(\gamma(t+\eta))-\Log_{\gamma(t)}(\gamma(t))}{\eta},\frac{\Log_{\gamma(t)}(\gamma(t+\eta))-\Log_{\gamma(t)}(\gamma(t))}{\eta}
    \right)}\\
    & \longrightarrow_{\eta \xrightarrow{} 0} \sqrt{g_{\gamma(t)}(\ud
        \Log_{\gamma(t)}\gamma(t)\cdot \gamma'(t),\ud \Log_{\gamma(t)}
    \gamma(t)\cdot \gamma'(t))} = \sqrt{g_{\gamma(t)}(\gamma'(t),\gamma'(t))}
  \end{align*}
\end{proof}

\section{Extremal points of the Benamou-Brenier Energy (Riemannian case)}
\label{app:extreme}

\begin{lemma}
  The functional $w : \gamma \in \Gamma^{\mathcal{M}} \mapsto
  \int_0^1 \alpha + \beta \left\| \gamma'(t) \right\|_g \ud t$, is
  lower semi-continuous.
\end{lemma}

\begin{proof}

  $w$ may be rewritten as the supremum of a family of continuous
  functions of $\gamma$ :

  \[w(\gamma) = \sup_{t_0=0<t_1\dots <t_n=1} \alpha +
    \beta\sum_{k=0}^{n-1} \frac{d_g(\gamma(t_k),
  \gamma(t_{k+1}))}{t_{k+1} - t_k},\]

  thus ensuring lower semi-continuity of $w.$

\end{proof}

\begin{lemma}
Let $w : \Gamma \xrightarrow{} ]0,+\infty]$ a lower semi-continuous
function and $D = \left\{\sigma \in \mathcal{M}(\Gamma)/
\int_{\Gamma} w \ud \sigma \leq 1\right\}.$
D is closed and:

\begin{align*}
\mathrm{Ext}(D) = \{0\}\cup \left\{ w(\gamma)^{-1}\delta_{\gamma} \,
| \, w(\gamma) < +\infty \right\}.
\end{align*}

\end{lemma}

\begin{proof}
For any convex decomposition $\sigma = \lambda \sigma_1 + (1-\lambda)
\sigma_2,$ if $\sigma = w(\gamma)^{-1}\delta_{\gamma}$ for any
$\gamma \in \Gamma,$ then $\sigma_1$ and $\sigma_2$ are supported
only on $\gamma$, \ie\ there exists $\alpha, \beta \in \R_+$ such
that $\quad \sigma_1 = \frac{\alpha}{w(\gamma)}\delta_{\gamma}, \quad
\sigma_2 = \frac{\beta}{w(\gamma)}\delta_{\gamma}$.\\
And then, as $\int w \ud \sigma = 1 = \lambda \alpha + (1-\lambda)
\beta$ and $\alpha + \beta = 1,$ we have $\alpha=\beta=1$ and
$\sigma=\sigma_1=\sigma_2$. Hence, $w(\gamma)^{-1}\delta_{\gamma}$ is
indeed an extremal point of $D.$\\

Let $\sigma \in \mathrm{Ext}(D),$ and make the assumption that there
are at least 2 curves $\gamma_0$ and $\gamma_1$ in the support of
$\sigma.$ Let $\Gamma_0 = \{d_{\Gamma}(\gamma,\gamma_0)\leq
\frac{1}{2}d_{\Gamma}(\gamma_1,\gamma_0)\},$ and $\Gamma_1 =
\Gamma_0^c,$ defining $\sigma_0 = \mathbbm{1}_{\Gamma_0} \sigma$ and
$\sigma_1 = \sigma - \sigma_0.$\\
If $\int w\ud \sigma_i\neq 0$ and $\int w \ud \sigma_i \leq \int w
\ud \sigma \leq 1,$ then with $\lambda = \int w \ud \sigma_0,$ we get
that $\sigma = \lambda \frac{\sigma_0}{\lambda} + (1-\lambda)
\frac{\sigma_1}{1-\lambda}$ and we have written $ \sigma$ as a convex
combination of elements of $D,$ so either $\sigma \notin
\mathrm{Ext}(D)$ or for a certain $i\in\{0,1\}, \int w\ud \sigma_i
=0.$ In the last case we have that by positivity of $w,$ $\sigma=0$
and then we can't have $\gamma_i \in \Gamma_i,$ which contradicts the
assumption.
Thus, we have that $\sigma$ has at most $1$ curve in its support
which means either $\sigma= 0$ or $\sigma = \frac{1}{w(\gamma)}
\delta_{\gamma}$ for some $\gamma \in \Gamma.$
\end{proof}

\section{Reeds-Shepp metric tensor and its affine connection}
\label{app:reeds-shepp}

Let $p = (x,y,\theta) \in \rototrans_2$ a rototranslation, the metric
tensor $g$ at $p$ is defined by:

\begin{align*}
g_{ij} =
\begin{bmatrix}
\cos(\theta)^2 + \varepsilon^{-2}\sin(\theta)^2 & \left( 1 -
\varepsilon^{-2}\right) \cos(\theta)\sin(\theta) & 0 \\
\left( 1 - \varepsilon^{-2}\right) \cos(\theta)\sin(\theta) &
\sin(\theta)^2 + \varepsilon^{-2}\cos(\theta)^2 & 0 \\
0 & 0 & \xi^2
\end{bmatrix},
\end{align*}

with constants $\varepsilon, \xi > 0$ while its contravariant inverse reads:

\begin{align*}
g^{ij} =
\begin{bmatrix}
\cos(\theta)^2 + \varepsilon^{2}\sin(\theta)^2 & \left( 1 -
\varepsilon^{2}\right) \cos(\theta)\sin(\theta) & 0 \\
\left( 1 - \varepsilon^{2}\right) \cos(\theta)\sin(\theta) &
\sin(\theta)^2 + \varepsilon^{2}\cos(\theta)^2 & 0 \\
0 & 0 & \xi^{-2}
\end{bmatrix}.
\end{align*}

The Christoffel symbols, up to the symmetries\footnote{The
Levi-Civita connection $\connex$ is \emph{by definition}
torsion-free.} $\Gamma^k_{ij} = \Gamma^k_{ji}$ and the null symbols,
of the Levi-Civita connection $\connex$ associated to $g$ write down:
\begin{align*}
\Gamma^x_{x,\theta} &= -\frac{1}{2\varepsilon^2}(\varepsilon^4 -
1)\cos(\theta)\sin(\theta) \\
\Gamma^x_{y,\theta} &= -\frac{1}{2\varepsilon^2}(\varepsilon^4 -
(\varepsilon^4 - 1)\cos(\theta)^2 - \varepsilon^2)\\
\Gamma^y_{x,\theta} &= \frac{1}{2\varepsilon^2}((\varepsilon^4 -
1)\cos(\theta)^2 - \varepsilon^2 + 1) \\
\Gamma^y_{y,\theta} &= \frac{\varepsilon^4 -
1}{2\varepsilon^2}\cos(\theta)\sin(\theta) \\
\Gamma^\theta_{x,x} &= \frac{\varepsilon^2 - 1}{\varepsilon^2\xi^2}
\cos(\theta)\sin(\theta)\\
\Gamma^\theta_{x,y} &= - \frac{1}{2\varepsilon^2\xi^2}
(2(\varepsilon^2 - 1)\cos(\theta)^2 - \varepsilon^2 + 1) \\
\Gamma^\theta_{y,y} &= - \frac{\varepsilon^2 -
1}{\varepsilon^2\xi^2}\cos(\theta)\sin(\theta).
\end{align*}

\section{Riemannian optimization on manifold and tangent bundle}
\label{app:riemannian-optim}

\subsection{Improved gradient descent by the Riemannian structure}

For first-order optimization, there is a simple adaptation of
gradient descent for Riemannian manifold $(\M, g)$.
The Riemannian gradient descent simply consists in a gradient
descent, but we take into account the non-Euclidean nature of the
manifold by leveraging the Riemannian exponential map to ramp the
update on the manifold.
The procedure is described in Algorithm \ref{algo:riemann}.

\begin{algorithm}[ht!]
\SetAlgoLined
\KwData{Objective function $f:x\in\M \mapsto f(x)\in \R$,
initialisation point $x_0 \in \M$, number of timesteps  $T \in \N^*$}

\For{$t = 0, 1, 2, \ldots, T-1$}{
Compute the Riemannian gradient $\nabla_{\M} f(x)$.

Update the iterate by:
\begin{align*}
  x = \Exp_{x}(-\eta \nabla_{\M} f(x))
\end{align*}
where $\text{Exp}_x$ denotes the Riemannian exponential map at point $x$.
}
\KwResult{Point $x$ of the manifold}

\label{algo:riemann}
\caption{Riemannian gradient descent}
\end{algorithm}

Usually, we replace the use of the exponential map by a
\emph{retraction}, i.e. a first-order approximation of this mapping.

\subsection{Optimisation on the tangent bundle}

Now, we want to carry out the Riemannian gradient descent not only on
the base manifold itself but also on the tangent bundle, which means
that the objective function is now the mapping:

\begin{align*}
f:(x,v) \in \M\times \Tang_x\M \mapsto f(x,v) \in \R.
\end{align*}

The tangent bundle $\TM$ is itself a $2d$--dimensional manifold with
tangent bundle $\TTM.$ For any base point $(p,u)\in \TM,$
$\Tang_{(p,u)}\TM$ has a particular structure \cite{Sasaki1958} as it
can be decomposed into horizontal and vertical subspaces, both $d-$dimensional:

\begin{align*}
\Tang_{(p,u)}\TM = \mathcal{H}_{(p,u)}\oplus\mathcal{V}_{(p,u)},
\end{align*}

There exists a natural lift from $\Tang_p\M$ to the horizontal and
vertical subspaces, let $X\in \Tang_p\M$ we denote $X^h$ the
horizontal part and $X^v$ the vertical part of $X$.

A natural choice of metric for the tangent bundle manifold is called
the Sasaki metric \cite{Sasaki1958}. It simply consists in using the
base metric and imposing perpendicularity of the horizontal and
vertical subspaces, \ie\ ensuring that the horizontal and vertical
subspaces remain orthogonal.

\begin{definition}
The \emph{Sasaki metric} $\sasaki$ is a metric on $\Tang_{(p,u)}\TM$
derived from the base metric $g$ by the following properties.

\begin{enumerate}
\item Horizontal reservation of the metric:
  \begin{math}
    \sasaki_{(p,u)}(X^h,Y^h) = g_p(X,Y),
  \end{math}
\item Vertical preservation of the metric:
  \begin{math}
    \sasaki_{(p,u)}(X^v,Y^v) = g_p(X,Y),
  \end{math}
\item Orthogonality:
  \begin{math}
    \sasaki_{(p,u)}(X^h,Y^v) = 0,
  \end{math}
\end{enumerate}
\end{definition}

\subsection{Energy in the piecewise geodesic framework on the Sasaki metric}

For ease of implementation and instead of the straightforward
piecewise geodesic model, we consider a set of points on the Sasaki
manifold associated with the base Reeds-Shepp metric
$(x_k,v_k)_{1\leq k\leq n-1}$. The piecewise geodesic curve
interpolating the points corresponds to the condition, $\forall 1\leq
k\leq n-1$: $\Exp_{x_k}(v_k) = x_{k+1}$.

We bring slight modifications to the energy by simply adding a
penalisation, to avoid the difficult and expensive computation of the
Riemannian Log map.

Finally, the energy we minimise in the oracle step writes down, for
$\lambda >0$:

\begin{align*}
F((x_k,v_k)) &= -\crogamma{\gamma_{(x_k,v_k)}}{\eta^{[k]}} + \beta
\sum_k (t_{k+1}-t_k) \, \norm{v_k}_{x_k}^2\\
& \quad + \lambda \sum_k \norm{\Exp_{x_k}(v_k) - x_{k+1}}_2^2.
\end{align*}

\end{document}